\documentclass[12pt]{amsart}

\usepackage{amsmath}
\usepackage{amssymb}
\usepackage[all]{xy}

\numberwithin{equation}{section}

\newtheorem{thm}{Theorem}[section]
\newtheorem{lem}[thm]{Lemma}
\newtheorem{prop}[thm]{Proposition}
\newtheorem{cor}[thm]{Corollary}
\newtheorem{rem}[thm]{Remark}

\newtheorem{exam-nota}[thm]{Example-Notation}

\newtheorem{dfn}[thm]{Definition}

\newtheorem{dfn-nota}[thm]{Definition-Notation}

\newtheorem{dfn-lem}[thm]{Lemma-Definition}

\newcommand{\beqa}{\begin{eqnarray*}}
\newcommand{\eeqa}{\end{eqnarray*}}

\newcommand{\la}{\mbox{$\langle$}}
\newcommand{\ra}{\mbox{$\rangle$}}

\newcommand{\qij}{q_{i,j}}
\newcommand{\tlambdaij}{{\tilde{\lambda}}_{i,j}}

\newcommand{\id}{\mbox{${\rm id}$}}

\newcommand{\lspan}{\mbox{${\rm span}$}}

\newcommand{\fe}{\mbox{${\mathfrak e}$}}
\newcommand{\fa}{\mbox{${\mathfrak a}$}}

\newcommand{\fk}{\mbox{${\mathfrak k}$}}
\newcommand{\fg}{\mbox{${\mathfrak g}$}}

\newcommand{\fl}{\mbox{${\mathfrak l}$}}
\newcommand{\fs}{\mbox{${\mathfrak s}$}}

\newcommand{\fh}{\mbox{${\mathfrak h}$}}

\newcommand{\fp}{\mbox{${\mathfrak p}$}}

\newcommand{\fb}{\mbox{${\mathfrak b}$}}
\newcommand{\fz}{\mbox{${\mathfrak z}$}}

\newcommand{\fu}{\mbox{${\mathfrak u}$}}

\newcommand{\xibar}{\mbox{$\overline{\xi}$}}

\newcommand{\C}{\mbox{${\mathbb C}$}}

\newcommand{\lara}{\la \, , \, \ra}

\newcommand{\Ad}{{\rm Ad}}

\newcommand{\Co}{\mathbb{C}}

\newcommand{\ghat}{{\hat{\mathfrak{g}}_{\mathcal{D}}}}

\newcommand{\fgl}{\mathfrak{gl}}

\newcommand{\xifij}{\xi_{f_{i,j}}}

\newcommand{\hfa}{\hat{\fa}}

\newcommand{\dn}{{n\choose 2}}
\newcommand{\dnone}{{n+1\choose 2}}
\newcommand{\XD}{X_{\mathcal{D}}}
\newcommand{\Dtower}{\mathcal{D}}
\newcommand{\coverD}{\dot{D}}

\newcommand{\Dinc}{{\mathcal{D}}_{c}}

\newcommand{\gdotdtower}{{\dot{\fg}}_{\mathcal{D}}}
\newcommand{\gtildedtower}{{\tilde{\fg}}_{\mathcal{D}}}

\newcommand{\alphaD}{\alpha}

\newcommand{\ddt}{{\frac{d}{dt}}|_{t=0}}
\newcommand{\Tr}{{\rm Tr}}
\newcommand{\piD}{{\pi}_{\mathcal{D}}}
\newcommand{\PiD}{{\Pi}_{\mathcal{D}}}
\newcommand{\nablaf}{{\nabla f}}

\newcommand{\ad}{\operatorname{ad}}
\newcommand{\ZD}{Z_{\mathcal{D}}}
\newcommand{\thetaD}{\theta_{\mathcal{D}}}
\newcommand{\qD}{q_{\mathcal{D}}}

\newcommand{\fzgen}{\fz_{\mathcal{D},gen}}
\newcommand{\fzsing}{\fz_{\mathcal{D},s}}
\newcommand{\uz}{z_{1},\dots, z_{n}}

\newcommand{\ghathess}{(\ghatgen)_{\mathfrak{b}+e}}
\newcommand{\kfibre}{\kappa^{-1}(\uz)}

\newcommand{\fzD}{\fz_{\mathcal{D}}}
\newcommand{\XDgen}{X_{\mathcal{D},gen}}
\newcommand{\ghatgen}{\hat{\fg}_{\mathcal{D},gen}}
\newcommand{\ghatsing}{\hat{\fg}_{\mathcal{D},s}}
\newcommand{\thetabarD}{\overline{\theta_{\Dtower}}}

\begin{document}

\title[]{On Algebraic Integrability of Gelfand-Zeitlin fields}

\author[M. Colarusso]{Mark Colarusso}
\address{Department of Mathematics, University of Notre Dame, Notre Dame, 46556}
\email{mcolarus@nd.edu}

\author[S. Evens]{Sam Evens}
\address{Department of Mathematics, University of Notre Dame, Notre Dame, 46556}
\email{sevens@nd.edu}

\maketitle

\begin{abstract}
We generalize a result of Kostant and Wallach concerning the algebraic integrability of the Gelfand-Zeitlin vector fields to the full set of strongly regular elements in $\fgl(n,\C)$.  We use decomposition classes to stratify the strongly regular set by subvarieties $\XD$.  We construct an \'{e}tale cover $\ghat$ of $\XD$ and show that $\XD$ and $\ghat$ are smooth and irreducible.  We then use Poisson geometry to lift the Gelfand-Zeitlin vector fields on $\XD$ to Hamiltonian vector fields on $\ghat$ and integrate these vector fields to an action of a connected, commutative algebraic group.


\end{abstract}

\section{Introduction}
\label{sec_intro}

In a series of papers \cite{KW1, KW2}, Kostant and Wallach study the action of a complex Lie group $A$ on $\fg=\fgl(n,\C)$.  The group $A$ is the simply connected, complex Lie group corresponding to the abelian Lie algebra $\fa$ generated by the Hamiltonian vector fields of the Gelfand-Zeitlin collection of functions.  The Gelfand-Zeitlin collection of functions contains $\frac{n(n+1)}{2}$ Poisson commuting functions and its restriction to each regular adjoint orbit forms an integrable system.  For each function in the collection, the corresponding Hamiltonian vector field is complete.  The action of $A$ on $\fg$ is then defined by integrating the Lie algebra $\fa$. 


Kostant and Wallach consider a Zariski open subset of $\fg$, called
the set of strongly regular elements, which consists of all elements where the
differentials of the Gelfand-Zeitlin functions
 are linearly independent.  The $A$-orbits of strongly regular elements are of dimension $\dn$ and form Lagrangian submanifolds of regular adjoint orbits. We denote by $x_i$ the upper left $i\times
i$ corner of the matrix $x \in \fg$.
Kostant and Wallach consider the Zariski open subset of strongly regular 
elements $M_{\Omega}(n)$ consisting of
$x \in \fg$ such that each $x_i$ is a regular semisimple element
of $\fg_i$ and $x_i$ and $x_{i+1}$ have no common eigenvalues. 
In \cite{KW2}, they show that there exists a covering $M_{\Omega}(n,\fe) \to M_{\Omega}(n)$
 such that the Lie algebra
$\fa$ lifts to $M_{\Omega}(n, \fe)$ 
and integrates to an algebraic action of a torus. Our purpose
in this paper is to extend this algebraic integrability result 
 to the full locus
of strongly regular elements.  More precisely, we stratify the strongly
regular set by smooth subvarieties, and for each stratum we construct a covering such that the Lie algebra $\fa$ lifts to the covering and integrates to an algebraic
action of a connected, abelian algebraic group.  


In more detail, the Gelfand-Zeitlin collection on $\fg$ is the collection of functions
$J_{GZ} = \{ f_{i,j}(x) : i=1, \dots, n, j=1, \dots, i \}$, where $f_{i,j}(x)= Tr((x_i)^j)$.
We denote
by $\fg_i = \{ x_i : x\in \fg \} \cong \fgl(i)$ embedded in $\fg$ as the upper left corner, 
and denote by $G_i \cong GL(i)$
the corresponding group. The
space $\fa$ spanned by $\{ \xi_f : f\in J_{GZ} \}$ is an abelian Lie algebra. An
element $x\in \fg$ is called {\it strongly regular} if $\{ df(x) : f\in J_{GZ} \}$ is
linearly independent. 
Kostant and Wallach showed that the set $\fg_{sreg}$ of
$\fg$ consisting of strongly regular elements is open and Zariski dense.  We stratify the strongly regular set using decomposition classes.  Let
$\fl_i$ be a standard Levi subalgebra of $\fg_i$ with blocks of sizes
$n_1, \dots, n_k$, let $\fz_i$ be the center of $\fl_i$, and let
$\fz_{i, gen}$ be the set of elements in $\fz_i$ with centralizer $\fl_i$.
 The regular decomposition class $D_i$ in $\fg_i$ consists of the subset
$G_i\cdot (\fz_{i, gen} + e^i)$, where
$e^i$ is the unique principal nilpotent element of $\fl_i$ in Jordan canonical form.  Thus, $D_{i}$ consists of all regular elements of $\fg_{i}$ whose Jordan form has blocks of sizes $n_{1},\dots, n_{k}$.  
Let $W^{L_i} = N_{G_i}(\fl_i)/L_i$,
which is a product of symmetric groups and acts on $\fl_{i}$ by permuting blocks of the same size. 
If $D_{i}\subset\fg_{i}$ $i=1,\dots, n$ is a sequence of regular decomposition classes, we call the sequence $\Dtower=(D_{1},\dots, D_{n})$ \emph{regular decomposition data}.  
 %

Let
\begin{gather*}
\XD = \{ x\in \fg : x_i \in D_i \} \cap \fg_{sreg},\\  
 \ghat = \{ (x,z_1, \dots,
z_n) \in \XD \times \Pi_{i=1}^n \fz_{i,gen} : x_i \in G_i \cdot (z_i + e^i)\}.
\end{gather*}
Consider the morphism
$\mu:\ghat \to \XD$ given by projection on the first factor, and let $\Sigma_{\Dtower} = \Pi_{i=1}^n
W^{L_i}$. 

\begin{thm} (Theorems \ref{thm_XDsmooth}, \ref{thm_ghatcover}
and \ref{thm_xdghatirr} )
\label{thm_intro1}
The morphism $\mu:\ghat \to \XD$ is a $\Sigma_{\Dtower}$-covering of smooth varieties. Further,
$\ghat$ and $\XD$ are connected.
\end{thm}

The variety $\XD$ is easily seen to be $A$-invariant, but the Lie algebra $\fa$ integrates to an action of an algebraic group on $\XD$ only for certain special choices of regular decomposition data $\Dtower$ (see Remark \ref{r:newerrem}).   

Consider the connected, abelian algebraic group $\ZD = \Pi_{i=1}^{n-1} Z_{G_i}(\fz_{i}+e^i)$.

\begin{thm}  (Theorem \ref{thm:ZDacts})
\label{thm_intro2}
The Lie algebra $\fa$ lifts to $\ghat$, and integrates to a free
algebraic action of $\ZD$ on $\ghat$.
\end{thm} 
Consider the open subsets  $\XDgen = \{ x\in \XD : x_i \ \text{and} \ x_{i+1} \ \text{have no common
eigenvalues}  \}$ and $\ghatgen = \mu^{-1}(\XDgen)$.  It is easily seen that $\ZD$ acts on $\ghatgen$.  In the special case where each $D_i$ consists of regular semisimple elements of $\fg_i$,
the covering $\ghatgen \to \XDgen$ coincides with the covering
$M_{\Omega}(n,\fe) \to M_{\Omega}(n)$ from \cite{KW2} and the group $\ZD=(\C^{\times})^{\dn}$ is a torus, and our result specializes to the
algebraic integrability result of Kostant and Wallach from \cite{KW2} (see Remark
\ref{rem_connecttokw}).  Thus, Theorem \ref{thm_intro2} generalizes this result to all of $\fg_{sreg}$, since $\fg_{sreg} = \bigcup \XD$, where the union is taken over all regular decomposition data $\Dtower$.  



 
The $\ZD$-action on $\ghat$ is a lift of a local $\ZD$-action defined on certain Lagrangian subvarieties of $\XD$ by the first author in Section 4 of \cite{Col1}. 
Let $\Phi: \fg_{sreg}\to \C^{\dnone}$ be the moment map for the Gelfand-Zeitlin integrable system. 
In \cite{KW1}, Kostant and Wallach show that for $x\in M_{\Omega}(n)$, $\Phi^{-1}(\Phi(x))$ is a single $A$-orbit and a homogeneous space for a free algebraic action of  $(\C^{\times})^{\dn}$.  
 In \cite{Col1}, the first author describes the action of $A$ on all strongly regular elements.  Further, in \cite{Col1}, for each sequence of regular decomposition data $\mathcal{D}$, an algebraic $\ZD$-action is constructed on the fibres $\Phi^{-1}(\Phi(x))$ for each $x\in\XD$.    
 The definition of the $\ZD$-action on $\Phi^{-1}(\Phi(x))$ makes use of the fact that the eigenvalues of each $x_{i}$, $i=1,\dots, n$ are constant on $\Phi^{-1}(\Phi(x))$.  This action of $\ZD$ cannot in general be extended to all of $\XD$, because there is no morphism $\XD\to \C^{r_{i}}$, $r_{i}=\dim\fz_{i}$ which assigns to $x_{i}$ a tuple of its eigenvalues in a prescribed order, except for certain regular decomposition data $\Dtower$ (see Remarks \ref{r:newrem} and \ref{r:newerrem}).  On the other hand, the covering $\ghat$ has a natural morphism which assigns to each $x_{i}$ an ordered tuple of eigenvalues, namely the morphism $\ghat\to \fz_{i,gen}$ which sends $(x, \uz)\to z_{i}$.  
This allows us to lift the local $\ZD$-action on $\Phi^{-1}(\Phi(x))$ to an algebraic action on $\ghat$.   
 


We construct the lift $\hfa$ of $\fa$ to $\ghat$ using Poisson geometry.
More precisely,
$\ghat$ is a subvariety of a Poisson variety $\gdotdtower$. Let $r_i = \dim (\fz_i)$
and let $s_i = i - r_i$.  For $(x, \uz)\in\ghat$, we use the semisimple part of the Jordan form of $x_{i}$ to construct $r_{i}$ functions $q_{i,j}$ on $\gdotdtower$ (Section \ref{sec_hamcenter}) and the nilpotent part of the Jordan form of $x_{i}$ to construct $s_{i}$ functions $p_{i,k}$ on $\gdotdtower$ (Section \ref{sec_hamsemisimple}).
%
We let $\hfa_i$ be the span of the Hamiltonian vector fields
$\xi_{q_{i,j}}$ and $\xi_{p_{i,k}}$, and we let $\hfa = \sum_{i=1}^{n-1} \hfa_i$. We show
that $\hfa$ is an abelian Lie algebra of vector fields tangent to $\ghat$ (Proposition
\ref{prop_ahatabelian}), and further
show that $\mu_*(\hfa)=\fa$ (Lemma \ref{l:ahat}). In addition, the vector fields  $\xi_{q_{i,j}}$ integrate to
give an algebraic $\Co^{\times}$-action on $\ghat$ (Proposition \ref{prop_algintegralfij}), 
and the vector fields $\xi_{p_{i,k}}$
integrate to give an algebraic $\C$-action on $\ghat$ (Proposition \ref{prop_algintpij}). These results imply that
$\hfa$ integrates to give an algebraic
action of $\ZD$ on $\ghat$.
 The $\ZD$-action on $\ghat$ is given
by a simple formula, and projects and specializes
to the more complicated 
 $\ZD$-action defined on the fibers $\Phi^{-1}(\Phi(x))$, $x\in\XD$. 
 We regard the simplicity of the
$\ZD$-action on $\ghat$ as a useful feature of the cover $\ghat$. In addition,
the cover $\ghat$ facilitates the use of Jordan decomposition to separate the
flows into semisimple and nilpotent parts, and makes Poisson computations easy
to do, as in \cite{KKS} and in \cite{EL}. 
 It would be interesting to relate
our work to work of Bielawski and Pidstrygach \cite{BP}, where a
different approach to the geometry of the Gelfand-Zeitlin action
is taken.

The paper is organized as follows. In Section \ref{sec_kwresults}, we recall
results from \cite{KW1} and \cite{KW2}, as well as the thesis of the
first author \cite{Col}, \cite{Col1}. In Section \ref{sec_dectowersmooth},
we recall facts about decomposition classes, show $\XD$ is smooth,
and construct the covering $\ghat \to \XD$. In Section \ref{sec_poisson},
we construct a Poisson structure on $\gdotdtower$ and compute its
anchor map. In Section \ref{sec_hamiltonian}, we prove the main results
of the paper.

In this paper, a variety is a complex quasi-projective algebraic set, 
and a subset of a variety is called a subvariety if it
is locally closed. The ring of regular functions on a variety $Y$ is
denoted $\C[Y]$. We use the Zariski topology on a variety
unless otherwise stated.

We would like to thank Michael Gekhtman, Bert Kostant, Hanspeter Kraft, Nolan Wallach, and Milen Yakimov for useful conversations relevant to the subject of
this paper. 
 The second author was partially supported by NSA grant 
 H98230-08-1-0023 during the preparation of this paper.

\section{Notation and results of Kostant, Wallach, and Colarusso}
\label{sec_kwresults}

Let $\fg=\fgl(n,\C)$ be the Lie algebra of $n \times n$ complex matrices.
For $i \le n$, let $\fg_{i} = \fgl (i,\C) \subset \fg$, regarded as 
the upper left $i\times i$ corner. If $x\in \fg$, let $x_i$ be
its upper left $i\times i$ corner of $x$, so the $kj$ matrix coefficient
$(x_i)_{kj}$ of $x_i$ is $(x)_{kj}$ 
  if $1 \le k, j \le i$,
and is zero otherwise. Let $G_i \cong GL(i,\C)$ be the closed
Lie subgroup of $GL(n,\C)$ with Lie algebra $\fg_{i}$.

The projection $\fg \to \fg_{i}$ given by $x\mapsto x_i$ induces
an injective ring homomorphism 
$\C [\fg_{i}] \to \C[\fg]$, which we use to regard
$\C [\fg_{i}]$ as a subalgebra of $\C[\fg]$.
In particular, we regard
 $J(n) = \C[\fg_{1}]^{G_1} \otimes_{\C} \dots \otimes_{\C} 
\C[\fg_{n}]^{G_n}$ as a subalgebra of
$\C[\fg]$.

For $i\le n$ and $j=1, \dots, i$, let $f_{i,j} \in \C [\fg]$
be the regular function defined by $f_{i,j}=Tr((x_i)^j)$.
Note that $f_{i,j} \in \C[\fg_{i}]^{GL(i)} \subset J(n)$.  Let $J_{GZ} = \{ f_{i,j} : 1\le i \le n, 1\le j \le i \}$.
Then $J_{GZ}$ freely generates the polynomial algebra $J(n)$.
Let $\fa =span \{ \xi_f : f\in J_{GZ} \}$, where $\xi_f$ is the
Hamiltonian vector field on $\fg$ defined by a function
$f\in \C[\fg]$ using the Lie-Poisson Poisson structure on
$\fg$.  It is shown in Theorem 3.25 of \cite{KW1} 
 that $J(n)$ is a maximal Poisson commutative subalgebra
of $\C[\fg]$. Thus, $\fa$ is an abelian Lie algebra, and further $\dim(\fa)=\dn$ (see \cite{KW1},
Section 3.2). Further, the vector fields $\xi_{f_{n,j}}=0$, so
$\fa$ is spanned by the vector fields $\xi_{f_{i,j}}$ with
$i \le n-1$. Let $A$
be the simply connected holomorphic Lie group 
with Lie algebra $\fa$. By Section 3 of \cite{KW1}, the group $A\cong \C^{\dn}$
integrates the action of $\fa$ on $\fg$. It follows from standard results in
symplectic geometry that $A\cdot x$ is isotropic in
the symplectic leaf $G\cdot x$ in $\fg$.

By definition, $x\in \fg$ is called {\it strongly regular}
if the set $\{ df(x) : f\in J_{GZ} \}$ is linearly independent
in $T^*_x(\fg)$. Let $\fg_{sreg}$ be the set of strongly regular
elements of $\fg$ and let $\fg_{reg}$ be the set of regular
elements of $\fg$, i.e., the set of elements whose centralizer
$\fz_{\fg} (x)$ has dimension $n$. By a well-known result of Kostant \cite{K},
 if $x \in \fg_{sreg}$, $x_k$ is regular for all $k$ (\cite{KW1}, Proposition
2.6).

We give alternate characterizations of the strongly
regular set in $\fg$. 


\begin{thm}\cite{KW1}
\label{thm_srchar}
Let $x\in \fg$. Then the following are equivalent.

\noindent (1) $x$ is strongly regular.

\noindent (2) $\dim(A\cdot x)=\dim(A) = \dn$ and $A\cdot x$ is
Lagrangian in $G\cdot x$.

\noindent (3) $x_{i}\in\fg_{i}$ is regular for all $i$, $1\leq i\leq n$ and $\fz_{\fg_{i}}(x_{i})\cap \fz_{\fg_{i+1}}(x_{i+1})=0$ for all $1\leq i\leq n-1$.

\end{thm}

Further, let $\fb$ be the upper triangular Borel subalgebra of $\fg$,
and let $e$ be the standard level one regular nilpotent in the opposite Borel to $\fb$,
i.e., $(e)_{k+1, k}=1$, and $(e)_{l, k}=0$ if $l \ne k + 1$.
 Then
$\fb+e$ is contained in the strongly regular set $\fg_{sreg}$.  Elements in the variety $\fb+e$ are called (upper) Hessenberg matrices. It follows
that $\fg_{sreg}$ is Zariski open and dense in $\fg$.

For $i=1, \dots, n$, consider the morphisms
$$
\Phi_i: \fg_i \to \C^i,\; \Phi_i(y)=(p_{i,1}(y), \dots, p_{i,i}(y)),
$$
where $p_{i,j}(y)$ is the coefficient of $t^{j-1}$ in the characteristic polynomial of $y$.  Note that $p_{i,j}\in \C[\fg_{i}]^{GL(i)}$.  Define 
$$
\Phi: \fg \to \C^1 \times \dots \times \C^n = \C^{\dnone}, \
\Phi (x)=(\Phi_1(x_1), \dots, \Phi_n(x_n)).
$$
Then the Kostant-Wallach map 
$\Phi:\fb+e \to \C^{\dnone}$ is an isomorphism of varieties
 (\cite{KW1}, Theorem 2.3). Hence, for $c\in \C^{\dnone}$, 
$\Phi^{-1}(c)_{sreg} := \Phi^{-1}(c) \cap \fg_{sreg}$ is
nonempty and open. By Proposition 3.6 in \cite{KW1}, $A\cdot x \subset
\Phi^{-1}(c)$. 

For $x\in \fg_i$, let $\sigma_i(x)$ equal the collection of
$i$ eigenvalues of $x$ counted with repetitions,
 where here we regard $x$ as an $i\times i$ matrix.

\begin{rem}
\label{rem_git}
If $x, y \in \fg$, then $\Phi(x)=\Phi(y)$ if and only if
$\sigma_i(x_i)=\sigma_i(y_i)$ for $i=1, \dots, n$.
\end{rem}
Let $c_{i}\in\C^{i}$ and consider $c=(c_{1}, c_{2},\dots, c_{n})\in\C^{1}\times\C^{2}\times\dots\times\C^{n}=\C^{\dnone}$.  Regard $c_{i}=(z_{1},\dots, z_{i})$ as the coefficients of the degree $i$ monic polynomial 
 \begin{equation}\label{eq:polyci}
 p_{c_{i}}(t)=z_{1}+z_{2} t+\dots + z_{i} t^{i-1} +t^{i}.
 \end{equation}

\begin{thm} (\cite{Col1}, Theorem 5.11)\
\label{thm:bigthm}
\noindent  Let $c=(c_{1}, c_{2},\dots, c_{n})
\in\C^{1}\times\C^{2}\times\cdots\times\C^{n}=\C^{\dnone}$, and
suppose that $p_{c_i}(t)$ and $p_{c_{i+1}}(t)$ have exactly $j_i$
roots in common. Then there are exactly $2^{\sum_{i=1}^{n-1} j_{i}}$
distinct $A$-orbits in $\Phi^{-1}(c)_{sreg}$. 
 For $x\in\Phi^{-1}(c)_{sreg}$, let $Z_{D_i}$ denote the centralizer 
of the Jordan form of $x_{i}$ in $\fg_{i}$, and consider the abelian
connected algebraic group $\ZD = Z_{D_1}\times\dots\times Z_{D_{n-1}}$.
Then $\ZD$ acts freely and algebraically on $\Phi^{-1}(c)_{sreg}$, and  
 the $A$-orbits on $\Phi^{-1}(c)_{sreg}$ 
coincide with the $\ZD$-orbits on
$\Phi^{-1}(c)_{sreg}$.


\end{thm}

\section{Decomposition classes and decomposition towers}
\label{sec_dectowersmooth}

\subsection{Decomposition classes}
\label{sec_decclass}

We recall some results about decomposition classes. See the papers
\cite{Brexp} and \cite{Br} of Broer
for details.
Let $\fl$ be a Levi subalgebra of a reductive Lie algebra $\fk$, 
let $\fz$ be the center of $\fl$, and let $W^L := N_K(\fl)/L$, where
$K$ is the adjoint group of $\fk$.
 Let $\fz_{gen}=\{ z \in \fz :
\fz_{\fk}(z)=\fl\}$, and let $\mathcal{O}_{x}$ be the $L$-orbit through
nilpotent $x \in \fl$. We say that the decomposition class in $\fk$ associated
to $\fl$ and $x$ is the set

$$
D(\fl, x) := K\cdot (\fz_{gen} + \mathcal{O}_{x})=K\cdot (\fz_{gen} + x).
$$

By \cite{Br}, Proposition 2.3, the morphism
$$
 K \times_{N_K(\fl)} (\fz_{gen} + N_K(\fl)\cdot x)
\cong D(\fl,x),\;  (k,y) \mapsto {\Ad}(k)y$$
is an isomorphism, so in particular $D(\fl, x)$ is
smooth.
Let
$$\coverD(\fl,x) := K \times_L (\fz_{gen} + N_K(\fl)\cdot x),$$
and consider the {\'e}tale morphism $\mu:\coverD(\fl,x)
\to D(\fl,x)$ given by $\mu(g, y) = \Ad(g)y$.
By \cite{Br}, Proposition 2.3(iii), the surjective morphism
$\theta: \coverD(\fl,x) \to \fz_{gen}$ given by $\theta(g,z+y)=z$
for $g\in G, z\in \fz_{gen}$ and $y\in N_K(\fl)\cdot x$ descends
to give a surjective morphism $\overline{\theta}:D(\fl,x)\to \fz_{gen}/W^L$.
Denote by $q:\fz_{gen} \to \fz_{gen}/W^L$ the quotient morphism.


\begin{prop}(See \cite{Br}, Proposition 2.3)
\label{prop_coveridentify}
Let $D(\fl,x)$ be a decomposition class.

\par\noindent (1) The following diagram is Cartesian with {\'e}tale horizontal maps:
\begin{equation}
\label{eq_cdiagdflx}
\begin{array}
{ccc}
\coverD(\fl,x) & \stackrel{\mu}{\to} & D(\fl,x)\\
\downarrow{\theta}  &   & \downarrow{\overline{\theta}} \\
\fz_{gen} & \stackrel{q}{\to} & \fz_{gen}/W^L.
\end{array}
\end{equation}


\noindent In particular, $\coverD(\fl,x)$ is smooth.

\par\noindent (2) $\overline{\theta}(D(\fl,x))=\fz_{gen}/W^L$ is smooth.
\end{prop}

\begin{rem}
\label{rem_dinc}
Let $$\Dinc(\fl,x) =  D(\fl,x) \times_{{\fz}_{gen}/W^L} \fz_{gen}
= \{(y,z):y\in K\cdot (z+N_{K}(\fl)\cdot x),\, z\in \fz_{gen} \}.$$
By Proposition \ref{prop_coveridentify}, the morphism
$\beta: \coverD(\fl,x) \to \Dinc(\fl,x)$ given by $\beta (g,z+y) =
(\Ad(g)(z+y), z)$ for $g\in K$, $z \in \fz_{gen}$, and $y\in N_K(\fl)\cdot x$
is an isomorphism. 

\end{rem}

\begin{rem}
\label{rem_GITid}
Let $r$ be the rank of $\fk$ and let $f_1, \dots, f_r$ be
algebraically independent generators of $\C [\fk]^K$, and
consider the morphism $F:\fk \to \C^r$ given by
$F(x)=(f_1(x), \dots, f_r(x))$. Then if $\chi:\fk \to \fk //K$
is the geometric invariant theory quotient, there is an
induced isomorphism $\overline{F}: \fk //K \to \C^r$.  Let $\fh\subset\fg$ be a Cartan subalgebra containing $\fz_{gen}$ and let $W=N_{G}(\fh)/H$ be the Weyl group.  Then,
the diagram 

$$
\begin{array}{ccc}
D(\fl,x) & \stackrel{\chi}{\to} & \fk //K \\
\downarrow{\overline{\theta}}  &   & \uparrow \\
\fz_{gen}/W^L & \stackrel{i}{\to} & \fh/W.
\end{array}
$$
commutes, where $i$ is the embedding induced by the inclusion
$\fz_{gen} \to \fh$, and the right vertical arrow is the
Chevalley isomorphism. As a consequence, $F(D(\fl, x)) \cong
\chi(D(\fl, x)) \cong \overline{\theta}(D(\fl, x))$ is smooth
by Proposition \ref{prop_coveridentify}. We will apply these ideas
later with $\fk = \fg_i$, and $F=\Phi_i:\fg_i \to \C^i$ as in Section \ref{sec_kwresults}. 

\end{rem}


\begin{rem}
\label{rem_regulardecomp}
We say a decomposition
class $D$ is {\it regular} if it is contained in $\fk_{reg}$,
the set of regular elements of $\fk$. $D(\fl, x)$ is regular if and only
if $x$ is regular nilpotent in $\fl$. In the remainder of this paper,
we consider only regular $D(\fl, x)$, in which case
$N_K(\fl)\cdot x = L\cdot x$.
\end{rem}

\begin{rem}
\label{rem_decomppartitition}
In the case when $\fk=\fg\fl (n)$,  a regular decomposition class
$D$ corresponds to a partition of $n$. Indeed, if $\lambda
= (n_1, \dots, n_r)$ is a partition of $n$, we associate to
$\lambda$ the Levi subalgebra $\fl$ consisting of block diagonal
matrices with blocks of size $n_j$ for $j=1, \dots, r$.
The corresponding decomposition class consists of 
 matrices 
conjugate to a block diagonal matrix $M(\lambda)$ with blocks of size
$n_j$ for $j=1, \dots, r$, where the $jth$ block $M_j = z_j \id_{n_j}
+ x^j$, where $\id_{n_j}$ is the $n_j \times n_j$
identity matrix, and $x^j$ is a regular nilpotent element
of $\fg\fl (n_j)$, and $z_i \not= z_j$ if $i\not= j$. 
It is elementary to show that every regular
element of $\fg\fl (n)$ is in $M(\lambda)$ for some partition $\lambda$.
The group $W^L$ is a product of symmetric groups, given by permuting
blocks of the same size (\cite{Br}, Section 9.1).
\end{rem}

\begin{lem}
\label{lem_orbitid}
Let $x, y \in \fg_{sreg}$. If $\Phi(x)=\Phi(y)$, then $y_i \in G_i \cdot
x_i$ for $i=1, \dots, n$.
\end{lem}

\begin{proof}
Since $\Phi_i(x)=\Phi_i(y)$, then $\chi(x_i)=\chi(y_i)$, where
$\chi:\fg\fl (i) \to \fg\fl (i)//GL(i)$ is the adjoint quotient.
The lemma now follows since each fiber of the adjoint quotient has
a unique regular conjugacy class (\cite{K}).
\end{proof}

Let $P$ be a parabolic subgroup of $K$
with Levi factor $L$ and unipotent radical $U$, 
 let $\fu$ and $\fp$ be the corresponding Lie algebras, and
let $\fl_1 = [\fl, \fl]$. 
Note that the quotient morphism
$\beta:\coverD(\fl, x) \to K\times_P (\fz_{gen} + L\cdot
x + \fu)$ is an isomorphism.
Surjectivity follows from the observation that if $z\in \fz_{gen}$
and $y$ is nilpotent in $ \fl_1$, then $U\cdot (z+y)=z+y+\fu$.  The reader
may verify this assertion using the fact that $z + y + \fu$ is an irreducible $U$-variety,
$U\cdot (x+y)$ is closed in $z+y + \fu$, and the stabilizer
$U_{z+y}$ is trivial. The remaining steps are routine to verify.
In particular, we regard $\coverD(\fl, x)$ as a locally closed
subvariety of $\dot{\fk} := K\times_P \fp$.

Let $\tilde{\fk} := K/P\times\fk$, and note that the
morphism $\alpha:\dot{\fk} \to \tilde{\fk}$ given by
$\alpha(g,y)=(gP, \Ad(g)y)$ is a closed embedding.



\subsection{Decomposition towers}
\label{sec_dectower}

For $i=1, \dots, n$, choose a regular decomposition class $D_i=G_i\cdot (\fz_{i,gen} + {\mathcal{O}}_{u_i})$,
with $\fz_{i,gen}$ the generic part of the center $\fz_i$ of the
Levi factor $\fl_i$ of $\fg_i$ determined by $D_i$, and
${\mathcal{O}}_{u_i}$ the regular nilpotent $L_i$-orbit in $\fl_i$.
Let $P_i$ be a parabolic containing $L_i$ for $i=1, \dots, n$.
We call the collection $\Dtower = (D_1, \dots, D_n)$ {\it regular
decomposition data}. Let 

$$\Sigma_{\Dtower} = W^{L_1} \times \dots \times W^{L_n}.$$
The group $\Sigma_{\Dtower}$ is a product of symmetric groups (see \cite{Br}, Section 9.1).
Let $$\fz_{\Dtower}: = \fz_{1, gen} \oplus \dots \oplus \fz_{n, gen}$$ and note that
the product action of $\Sigma_{\Dtower}$ on $\fz_{\Dtower}$ is free, so
$\fz_{\Dtower}/\Sigma_{\Dtower}$ is smooth of dimension $\dim(\fzD)$.

\begin{dfn}
\label{dfn_tower}
The subvariety $$\XD := \{ x \in \fg_{sreg} : x_i \in D_i \}$$ is called a tower of decomposition classes.
\end{dfn}







Recall the morphisms 
 $\overline{\theta_i} : D_i \to \fz_{i, gen}/W^{L_i}$
  from Diagram (\ref{eq_cdiagdflx}). Denote by
$\thetabarD :\XD \to \fzD/\Sigma_{\Dtower}$ the morphism
$\thetabarD (x) = (\overline{\theta_1}(x_1), \dots, \overline{\theta_n}(x_n))$. 
Let $\fh_i$ be a Cartan subalgebra of $\fg_i$ containing $\fz_i$, and let
$W_i$ be $N_{G_i}(\fh_i)/H_{i}$, the corresponding Weyl group.
Recall the embedding $\fz_{i, gen}/W^{L_i} \to \fh_i/W_i$ and the isomorphism
$\Phi_i:\fh_i/W_i \to \C^i$ from Remark \ref{rem_GITid}. They compose to give an embedding
$\overline{\Phi_i}:\fz_{i, gen}/W^{L_i} \to \C^i$. We consider the product morphism 
$$\overline{\Phi}=(\overline{\Phi_1}, \dots, \overline{\Phi_n}):\fzD/\Sigma_{\Dtower } \to \C^{\dnone}.$$

\begin{lem}
\label{lem_XDcart}
The diagram

$$
\begin{array}{ccc}
\XD & {\hookrightarrow} & \fg_{sreg} \\
\downarrow{\thetabarD}  &   & \downarrow{\Phi} \\
 \fzD/\Sigma_{\Dtower } & \stackrel{\overline{\Phi}}{\to} & \C^{\dnone}
\end{array}
$$

\noindent is Cartesian.  In particular, $\Phi^{-1}(\Phi(\XD))\cap\fg_{sreg}=\XD$, so that $\XD$ is a union of $A$-orbits in $\fg_{sreg}$.    
\end{lem}

\begin{proof} 
Let $z=(z_1, \dots, z_n) \in \fzD$, and denote by $\overline{z}$ its
representative in $\fzD/\Sigma_{\Dtower}$. If $\overline{\Phi}(\overline{z})
=\Phi(x)$ for $x\in \fg_{sreg}$, then $\Phi_i(z_i)=\Phi_i(x_i)$ for $i=1, \dots, n$.
Thus, $\Phi_i(z_i + u_i)=\Phi_i(x_i)$, so for each $i$, $x_i \in G_i\cdot (z_i + u_i)$
by the proof of Lemma \ref{lem_orbitid}.  It follows that $x\in \XD$ and that $\thetabarD(x)=\overline{z}$.  
\end{proof}

\begin{rem}\label{r:dims}
The surjectivity of the Kostant-Wallach map $\Phi:\fg_{sreg}\to\C^{\dnone}$ (Theorem 2.3 in \cite{KW1}) along with the argument in the proof of Lemma \ref{lem_XDcart} imply that the morphism $\thetabarD: \XD\to \fzD/\Sigma_{\Dtower}$ is surjective.  It then follows easily that $\dim \Phi(\XD)=\dim \fzD/\Sigma_{\Dtower}=\dim\fzD$.  

\end{rem}

\begin{thm}
\label{thm_XDsmooth}
The subvariety $\XD$ is a smooth subvariety of $\fg$, and all its connected components
have dimension $\dim(\fzD) + n^2 - \dnone$.
\end{thm}

\begin{proof}
By  Theorem 2.3 of \cite{KW1}, $\Phi:\fg_{sreg} \to \C^{\dnone}$ is
a surjective submersion, so $\Phi$ is smooth of relative dimension
$n^2 - \dnone$ by Proposition III.10.4 of \cite{Ha}. By 
 Proposition III.10.1(b) of \cite{Ha} and Lemma \ref{lem_XDcart}, 
the morphism $\thetabarD:\XD \to \fzD/\Sigma_{\Dtower}$ is
smooth of relative dimension $n^2 - \dnone$. By Remark \ref{r:dims},
 $\fzD/\Sigma_{\Dtower}$
is smooth of dimension $\dim(\fzD) = \dim(\Phi(\XD))$, and it follows from Proposition
III.10.1(c) of \cite{Ha} that $\XD$ is smooth of dimension
$n^2 - \dnone + \dim(\fzD)$. The result now follows from definitions.
\end{proof}


\subsection{Covers of decomposition towers}
\label{sec_ghat}

Fix regular decomposition data $\Dtower = (D_1, \dots, D_n)$
and associated notation, as in the last section. Let

$$
\gdotdtower := \Pi_{i=1}^n \dot{D}(\fl_i, u_i), \ \ \
\gtildedtower := \Pi_{i=1}^n G_i/P_i \times \fg_i.
$$

Consider the locally closed embedding
\begin{equation}
\label{eq_alphadef}
\alpha = (\alpha_1, \dots, \alpha_n):\gdotdtower
\to \gtildedtower,
\end{equation}
 where 
$\alpha_i (g_i, y^i)=(g_i P_i, \Ad(g_i)y^i)$ with $g_{i}\in G_{i}$ and $y_{i}\in\fz_{i,gen}+\mathcal{O}_{u_{i}}$.

Denote by $\mu=(\mu_1, \dots, \mu_n): \gdotdtower \to \Pi_{i=1}^n D_i$,
where $\mu_i(g_i, y^i) = \Ad(g_i)y^i$. Consider also the 
embedding $\gamma:\fg \to \Pi_{i=1}^n \fg_i$
given by $\gamma(x)=(x_1, \dots, x_n)$.

Let $\ghat := \XD \times_{\Pi_{i=1}^n D_i} \gdotdtower$,
so the diagram

\begin{equation}
\label{eq_ghatcartesian}
\begin{array}{ccc}
\ghat &\to & \gdotdtower \\
\downarrow{\mu}  &   & \downarrow{\mu} \\
\XD & \stackrel{\gamma}{\to} & \Pi_{i=1}^n D_i
\end{array}
\end{equation}
is Cartesian.  Note that the canonical morphism
 $\ghat \to \gdotdtower$ is a locally closed
 embedding, so we can view $\ghat$ as a subvariety of $\gdotdtower$.  



\begin{rem}
\label{rem_ghatcriterion}
A point $y=(g_1, y^1, \dots, g_n, y^n) \in \gdotdtower$ is
contained in $\ghat$ if and only if $(\Ad (g_n)y^n)_i = \Ad(g_i)y^i$
for $i=1,\dots, n-1$ and $\Ad (g_n)y^n \in \fg_{sreg}$.
\end{rem}

Recall the morphisms 
 $\theta_i : \dot{D}(\fl_i, u_i) \to \fz_{i, gen}, \; \overline{\theta_{i}}: D_{i}\to \fz_{i,gen}/ W^{L_{i}}$
  from Diagram (\ref{eq_cdiagdflx}).  Denote by
$\thetaD :\gdotdtower \to \fzD$ the morphism
$\thetaD  = (\theta_1, \dots, \theta_n)$.  
Abusing notation, we denote by $\thetabarD:\prod_{i=1}^{n} D_{i}\to \fzD/\Sigma_{\Dtower}$
 the morphism
$\thetabarD=(\overline{\theta_{1}},\dots, \overline{\theta_{n}})$.  We note that
$\thetabarD \circ \gamma:\XD \to \fzD/\Sigma_{\Dtower}$ coincides with
the morphism $\thetabarD:\XD \to \fzD/\Sigma_{\Dtower}$ from Section
\ref{sec_dectower}.

\begin{thm}
\label{thm_ghatcover}
The morphism 
$\mu: \ghat \to \XD$ is a $\Sigma_{\Dtower}$-covering,  $\ghat$ is smooth, and all its
connected components have dimension $\dim(\fzD) + n^2 - \dnone$.
\end{thm}

\begin{proof}
Consider the Cartesian diagram
\begin{equation}
\label{eq_ghatcover}
\begin{array}{ccc}
\gdotdtower & \stackrel{\thetaD}{\to} & \fzD \\
\downarrow{\mu}  &   & \downarrow{\qD} \\
 \Pi_{i=1}^n D_i & \stackrel{\overline{\thetaD}}{\to} & \fzD/\Sigma_{\Dtower},
\end{array}
\end{equation}

\noindent where $\qD$ is the quotient of the action of $\Sigma_{\Dtower}$.
This diagram is the product of  Cartesian diagrams from Equation
(\ref{eq_cdiagdflx}).
By Proposition \ref{prop_coveridentify}, the product morphism
$\mu$ is a $\Sigma_{\Dtower}$-covering, and the first claim follows easily
from the Cartesian diagram (\ref{eq_ghatcartesian}).
Smoothness of $\ghat$ and the dimension assertion follow from
Theorem \ref{thm_XDsmooth}. 
\end{proof}
\begin{rem}\label{r:newrem}
Let $\mathcal{D}=(D_{1},\dots, D_{n})$ with $D_{i}=(\fl_{i},u_{i})$ and suppose for each $i$ that all blocks of $\fl_{i}$ have different sizes (see Remark \ref{rem_decomppartitition}).  Then $\Sigma_{\mathcal{D}}$ is trivial and $\ghat\cong\XD$. 
\end{rem}

\begin{rem}
\label{ghatxdzsq}
We give another characterization of the variety 
$\ghat$, which will be useful in Sections \ref{sec_integration} and
\ref{sec_ghatgeneric}. Indeed, by Diagrams (\ref{eq_ghatcartesian}) 
and (\ref{eq_ghatcover}), it follows that
\begin{equation}
\label{eq_ghatxd}
\begin{array}{ccc}
\ghat & \stackrel{\thetaD}{\to} & \fzD \\
\downarrow{\mu}  &   & \downarrow{\qD} \\
 \XD & \stackrel{\overline{\thetaD}}{\to} & \fzD/\Sigma_{\Dtower}
\end{array}
\end{equation}
is Cartesian. Since $\XD \to \fzD/\Sigma_{\Dtower}$ is a surjective submersion, it
follows that $\ghat \to \fzD$ is a surjective submersion.
Hence,

$$
\ghat \cong \XD \times_{\fzD/\Sigma_{\Dtower}} \fzD = \Dinc,
$$
where $\Dinc := \{ (x, z_1, \dots, z_n) : 
x\in \fg_{sreg},\, x_i\in G_{i}\cdot (z_{i}+u_{i}) \}$.
We denote by $\mu:\Dinc \to \XD$ projection on $\XD$, 
and by  $\kappa:\Dinc\to\fzD$ the projection on $\fzD$, so 
$\kappa(x,\uz)=(\uz)$.
\end{rem}

\begin{rem}
\label{rem_connectedlater}
Recall the open subset $\XDgen$ of $\XD$ defined in the introduction,
and its preimage $\ghatgen = \mu^{-1}(\XDgen)$.
 In Remark \ref{rem_connecttokw}, in the special case
when all $D_i$ are regular semisimple,  we identify
the cover $\ghatgen \to \XDgen$ with the cover $M_{\Omega}(n,\fe) \to M_{\Omega}(n)$
that plays a key role in \cite{KW2}. In Theorem \ref{thm_xdghatirr}, we show
that $\ghat$ and $\XD$ are connected.
\end{rem}

\section{Poisson geometry of $\gdotdtower$ }
\label{sec_poisson}

In this section, we use standard results from Poisson geometry to
construct and compute a Poisson structure $\piD$ on $\gdotdtower$.

\subsection{Recollections from Poisson geometry} 
\label{sec_pfacts}

Let $(M,\pi)$ be a Poisson manifold. Thus, $\pi$ is a global
section of $ \wedge^2 (TM)$, and if $f, g$ are functions
on $M$, their Poisson bracket  defined by $\{ f, g \} :=
\pi (df, dg)$ makes the ring of functions on $M$ into a Poisson
Lie algebra.
Let $\widetilde{ \pi}:T_x^*(M) \to T_x(M)$ be the anchor map,
defined by $\widetilde{\pi}(\alpha)(\beta)=\pi(\alpha, \beta)$
if $\alpha, \beta$ are cotangent vectors at $x\in M$.
 For $f\in \C[M]$,
we let $\xi_f := \widetilde{\pi}(df)$ be the Hamiltonian
vector field of $f$. For $f,g \in \C[M]$,
$[\xi_f, \xi_g] = \xi_{\{ f, g \}}$.
 If $N \subset M$ is
a submanifold of $M$, its {\it characteristic distribution}
is $\widetilde{\pi}(T_N^*(M))$, where $T_N^*(M)$ is the conormal
bundle to $N$ in $M$.  

\begin{dfn}
\label{dfn_coisotropic}
A submanifold $N$ of $M$ is called coisotropic with respect to
$\pi$ if its characteristic distribution $\widetilde{\pi}(T_N^*(M))
\subset TN$.
\end{dfn}

Given two Poisson manifolds $(M, \pi_{M})$ and $ (R, \pi_{R})$ a smooth
 map $\phi: M\to R$ is  Poisson if $\phi_{*}\pi_{M}=\pi_{R}$.  

\begin{prop}
\label{prop_preduction}
Let $\phi:(M, \pi_M) \to (R, \pi_R)$ be a 
surjective Poisson submersion between Poisson manifolds.
Assume that

1)  $Q$ is a coisotropic submanifold 
of $(M, \pi_M)$;

2)  The characteristic distribution of $\pi_M$ on $Q$
is a subspace of
 the distribution defined by the tangent spaces to the fibers of $\phi$;

3) $\phi(Q)$ is a smooth submanifold of $R$; 

\par\noindent Then $\phi(Q)$ is a Poisson submanifold of
$(R, \pi_R)$.
\end{prop}

We remark that this Proposition is a mild generalization of
 Proposition 6.7 from
\cite{EL}, and the proof given in \cite{EL} also works for our
result here. 

\subsection{The Poisson structure on $\coverD(\fl,x)$} 
\label{sec_poissontildevariety}

A symplectic form $\omega$ on $M$ induces an identification
$\widetilde{\omega}:T(M) \to T^*(M)$ given by
$\widetilde{\omega}(\xi)(\eta)=\omega (\xi, \eta)$ for
vector fields $\xi, \eta$ on $M$. We consider the bivector
$\pi = \pi_\omega$ such that the second exterior power
of $\widetilde{\omega}$ maps $\pi$ to $\omega$. Then $(M, \pi)$
is Poisson (see Section 1.2 of \cite{CG}).

We apply this construction to the case where $M=T^*K$
is the cotangent bundle of a reductive group $K$ with
invariant nondegenerate bilinear form $\lara$ on its
Lie algebra $\fk$. We identify $T^*K = K\times \fk^* = K\times \fk$,
using left-invariant forms in the first identification and $\lara$ for
the second identification. 
Fix a point $(g,\alpha) \in M = K\times \fk$. 
\begin{dfn}
\label{dfn_vfgen}

(1) For $X\in \fk$, let $\xi_X$ be the tangent vector at $(g,\alpha)$ to the
curve $t\mapsto (g \exp tX, \alpha)$ at $t=0$;

(2) For $\beta \in \fk$, let $\eta_\beta$ be the tangent vector 
at $(g, \alpha)$ to
the curve $t\mapsto (g, \alpha + t\beta)$ at $t=0$.
\end{dfn}

We identify
$\fk + \fk \cong T_{(g,\alpha)}(M)$ via the map
$(X, \beta)\mapsto (\xi_X, \eta_\beta)$.
We identify the dual space $T^*_{(g,\alpha)}(M)=\fk + \fk$ using 
the form $\lara$.
Using these identifications, a symplectic form $\omega_{cl}$ on $K\times \fk$
is given by the formula (cf. p. 497 of \cite{KKS}):

 $${\omega_{cl}}_{(g,\alpha)}((X_1, \beta_1),(X_2, \beta_2))
= -< \beta_1, X_2 > + < \beta_2, X_1 > + < \alpha, [X_1, X_2]>.$$

Up to sign, $\omega_{cl}$ coincides with the canonical symplectic form
on the cotangent bundle from Section 1.1 of \cite{CG}.
We denote by $\pi_{cl} = \pi_{\omega_{cl}}$ the induced Poisson
bivector on $T^*K$. Then the anchor map of $\pi_{cl}$ is given by
the formula:

\begin{equation}
\label{eq_anchormap}
\widetilde{\pi_{cl}}_{(g,\alpha)}(\gamma, Y)=(-Y, \gamma + [Y,\alpha]), \ 
\gamma, Y \in \fk.
\end{equation}

Recall the notation of Section \ref{sec_decclass}, so
$\fp$ is a parabolic subalgebra of $\fk$ with
Levi decomposition $\fp=\fl+\fu$, and $\fz_{gen}$ is the generic part of
the center $\fz$ of $\fl$. 

Let $Q = K \times (z + L\cdot x + \fu)$, where $z\in \fz_{gen}$
and $x\in \fl$ is nilpotent. Note that $P$ acts on $Q$ diagonally
by $p \cdot (g,\alpha) = (gp^{-1}, {\Ad}(p) \alpha)$, for
$p\in P$, $(g,\alpha) \in Q$.

\begin{lem}
\label{lem_Qcoisotropic}
 The subvariety $Q$ is a coisotropic subvariety
of $M=K\times \fk$. Further, the characteristic distribution is
tangent to the diagonal $P$-action on $Q$.
\end{lem}

\begin{proof}
Let $(g,\alpha) \in Q$ and let $\alpha = z + y + u$, with
$y\in L\cdot x$ and $u\in \fu$.
Using the above identification $T_{(g,\alpha)}(M)=\fk + \fk$, 
$T_{(g,\alpha)}(Q)= (\fk, [\fl,y]+\fu)$, so that
$(T^{*}_{Q}(M))_{(g,\alpha)}=(0, ([\fl,y]+\fu)^{\perp})$. 
Note that
$$
([\fl,y]+\fu)^{\perp}=[\fl,y]^{\perp}\cap \fu^{\perp}=(\fu^{-}+\fz_{\fl}(y)+\fu)\cap \fp=\fz_{\fl}(y)+\fu,
$$
so that
\begin{equation}
\label{eq_idcotangent}
(T^{*}_{Q}(M))_{(g,\alpha)}=(0,\fz_{\fl}(y)+\fu ).
\end{equation}
Applying (\ref{eq_anchormap}), we compute the characteristic distribution
\begin{equation}
\label{eq_chardistcomputation}
\widetilde{\pi_{cl}}(T^{*}_{Q}(M))_{(g,\alpha)}=\{(-Y, [Y,\alpha]) :
Y \in \fz_{\fl}(y)+\fu \}.
\end{equation}
It follows that the characteristic distribution is in the tangent space to
the diagonal $P$-action.  
  In particular, $Q$ is coisotropic.
\end{proof}

Note that the diagonal $P$-action on $K\times \fk$ 
preserves the Poisson structure $\pi_{cl}$. 
Hence, if we consider the
projection map ${\rm pr}: K\times \fk \to K\times_P \fk,$
then $\pi:= {\rm pr}_* \pi_{cl}$ is a well-defined Poisson
structure on $K\times_P \fk$. Since $\alpha: K \times_P \fk \to K/P \times \fk$
is an isomorphism, it follows that $\Pi := \alpha_* \pi$
is a well-defined Poisson structure on $K/P \times \fk$.
Further, if we let $\phi = \alpha \circ {\rm pr}$ so
$\phi(g,\alpha)=(gP, \Ad(g) \alpha)$, then $\phi:(K\times \fk, \pi_{cl})
\to (K/P \times \fk, \Pi)$ is a Poisson morphism.

\begin{prop}
\label{prop_tildepoisson}
Let $D(\fl, x)$ be a decomposition class. Then
the variety $\coverD(\fl, x)$ is a Poisson
subvariety of $(K\times_P \fk, \pi)$.
 The morphism
$\mu:\coverD(\fl, x) \to \fk$ defined by $p(g,v)=\Ad(g) v$ is Poisson.
\end{prop}

\begin{proof}
Choose $z\in \fz_{gen}$.
We apply Proposition \ref{prop_preduction} with $M=K\times \fk$,
$Q= K\times (z + L\cdot x + \fu)$ and $R=K \times_P \fk$,
and the quotient morphism $\phi:M \to R$ defined above.

By Lemma \ref{lem_Qcoisotropic}
 and Proposition \ref{prop_preduction}, $K\times_P (z + L\cdot x + \fu)$ 
is a
Poisson subvariety of $(R, \pi)$.  We show that
that $\pi$ is nondegenerate on $K\times_P (z + L\cdot x + \fu)$.
 Note that the morphism $\mu:R \to \fk$ is Poisson by Lemma 1.4.2 and Proposition 1.4.10
of \cite{CG}, so its
restriction $\mu:K\times_P (z + L\cdot x + \fu) \to K\cdot (z + x)$
is a Poisson covering by Proposition \ref{prop_coveridentify}. Since $K\cdot (z + x)$ has nondegenerate
Poisson structure, it follows that $K\times_P (z + L\cdot x + \fu)$
is symplectic. Thus, $\coverD(\fl, x)$ is a union of symplectic
leaves, and hence Poisson. 
\end{proof}

\begin{rem}
\label{rem_symreduction}
The symplectic structure on $\coverD(\fl,x)$ can also be realized
by symplectic reduction (see \cite{KKS} and \cite{CG}).
\end{rem}

Since each $\coverD(\fl_i, u_i)$ has Poisson structure $\pi_i$,
the product $\gdotdtower = \Pi_{i=1}^n \coverD(\fl_i, u_i)$
inherits a product Poisson structure $\piD$. Further, the
product $\gtildedtower := \Pi_{i=1}^n \tilde{\fg_i}$ 
inherits a product Poisson structure $\PiD$, and the
product morphism $\alpha : (\gdotdtower, \piD) \to (\gtildedtower, \PiD)$
from Equation (\ref{eq_alphadef})
is Poisson.

\subsection{Computation of the anchor map }
\label{sec_anchor}

In this section, we compute the anchor map on $K/P \times \fk$
for $\Pi$.

Since $\phi:(K\times \fk, \pi_{cl}) \to (K/P \times \fk, \Pi)$
is Poisson, it follows easily that 
$\widetilde{\Pi}_{(gP, \Ad(g) \alpha)} = \phi_{(g, \alpha),*} \circ {\widetilde{\pi_{cl}}_{ (g, \alpha)}} \circ \phi_{(g, \alpha)}^*$.
We factor $\phi = (p, \mu)$, where $p(g,\alpha)=gP$ and
$\mu(g,\alpha)={\Ad}(g)\alpha$. As in the last section, we identify
$T_{(g,\alpha)}(K\times \fk) = \fk + \fk$.
We further identify

\begin{equation}
\label{eq_tanidentify}
\fk/\fp + \fk \cong T_{(gP,\beta)}(K/P \times \fk); \; 
(X + \fp, \gamma) \mapsto (\xibar_X, \eta_\gamma),
\end{equation}
where $\xibar_X$ is the tangent vector to the curve
$t\mapsto (g\exp(tX) P , \beta)$ at $t=0$, and $\eta_\gamma$
is the tangent vector to the curve $(gP, \beta + t\gamma)$
at $t=0$.
We  identify
$T_{(gP, \beta)}^*(K/P \times \fk) = \fu + \fk$, using
the identification $\fu = \fp^\perp$. 

It is routine to check that if $(X,\beta) \in \fk + \fk$
and $\gamma \in \fk = T_{\Ad(g)\nu}^* \fk$,
then the differential and codifferential are computed by
\begin{equation}
\label{eq_mustar}
\mu_{(g,\nu),*}(X, \beta)={\Ad}(g)([X, \nu] + \beta),\; \ 
\mu_{(g,\nu)}^* (\gamma)=
([\nu, \Ad(g^{-1})\gamma], \Ad(g^{-1})\gamma).
\end{equation}
 It follows that the differential
\begin{equation}
\label{eq_philowerstar}
\phi_{(g,\nu),*}(X, \beta)=(X + \fp, {\Ad}(g)([X,\nu] + \beta)),
\end{equation}
and hence the codifferential
\begin{equation}
\label{eq_phiupperstar}
\phi_{(g,\nu)}^*(\lambda, \gamma)=
(\lambda +[\nu,\Ad(g^{-1}) \gamma],  \Ad(g^{-1})\gamma),\;
\lambda \in \fu, \gamma \in \fk.
\end{equation}
Thus for $\lambda \in \fu$,
\begin{equation}
\label{eq_anchorfirst}
{\widetilde{\Pi}}_{(gP,\Ad(g)\zeta)}(\lambda, 0)=\phi_{(g,\zeta),*} 
\widetilde{\pi_{cl}}_{(g,\zeta)}(\lambda, 0)= 
\phi_{(g,\zeta),*}(0, \lambda)=(0, {\Ad}(g)\lambda),
\end{equation}
using Equations (\ref{eq_philowerstar}), (\ref{eq_phiupperstar}), and (\ref{eq_anchormap}).
Further, using the same equations, 
we obtain
\begin{equation}
\label{eq_anchorsecond}
\begin{split}
{\widetilde{\Pi}}_{(gP,\Ad(g)\zeta)}(0, \gamma)&=
\phi_{(g,\zeta),*} \widetilde{\pi_{cl}}_{(g,\zeta)}
([\zeta, {\Ad}(g^{-1}) \gamma], {\Ad}(g^{-1}) \gamma) \\
&= \phi_{(g,\zeta),*} (-{\Ad}(g^{-1}) \gamma, 0) 
= (-{\Ad}(g^{-1}) \gamma + \fp, [ \Ad(g)\zeta, \gamma]).
\end{split}
\end{equation}
It follows that
\begin{equation}
\label{eq_anchorboth}
{\widetilde{\Pi}}_{(gP,\Ad(g)\zeta)}(\lambda, \gamma) =
(-{\Ad}(g^{-1})\gamma + \fp, [\Ad(g)\zeta, \gamma] + \Ad(g)\lambda).
\end{equation}

\begin{rem}
\label{rem_leftright}
We can instead identify $\fk /\Ad(g) \fp \cong T_{gP}(K/P)$
by letting $X + \Ad(g) \fp$ be the tangent vector to
$t\mapsto \exp(tX)gP$ at $t=0$,
and use the corresponding identification $\Ad(g) \fu \cong T_{gP}^*(K/P)$. 
When we use these identifications, we obtain the formula
for the anchor map

\begin{equation}
\label{eq_anchorrightversion}
{\widetilde{\Pi}}_{(gP,\Ad(g)\zeta)}(\lambda, \gamma) =
(-\gamma + \Ad(g) \fp, [\Ad(g)\zeta, \gamma] + \lambda).
\end{equation}
With these identifications, it is clear that the anchor
map does not depend on the representative chosen for the
coset $gP$.
\end{rem}



\section{Lifting of Gelfand-Zeitlin fields to $\ghat$ and algebraic integrability}
\label{sec_hamiltonian}

In this section,
let $\Dtower  = (D_1, \dots, D_n)$ be regular decomposition data with
$D_i = D(\fl_i, u_i)$.  We introduce a collection of ${n\choose 2}$ functions
$\hat{J}$ on $\gdotdtower $ and compute their Hamiltonian
vector fields on $\ghat$. 
More precisely, if $D_i = D(\fl_i, u_i)$, then we define
$r_i = \dim(\fz_i)$ functions $q_{i,j}$ on $\gdotdtower$ using
$\fz_i$, and we define $s_i = i - r_i$ functons $p_{i,k}$ on $\gdotdtower$ using the semisimple part $[\fl_i, \fl_i]$ of $\fl_i$. The functions
$\hat{J_i} = \{ q_{i,j}, p_{i,k} : j=1, \dots, r_i, k=1, \dots, s_i \}$
give $i$ functions on $\ghat$, and we define $\hat{J} = \bigcup_{i=1}^{n-1} \hat{J_i}$.
We let $\hat{\fa} $ be the linear span of $\{ \xi_f : f\in \hat{J} \}$.
We further show that $\hat{\fa}$ is abelian, and the vector fields
in $\hat{\fa}$ lift the vector fields in $\fa$ on $\XD$ to $\ghat$
 and are algebraically
integrable on the $\Sigma_{\Dtower}$-cover $\ghat$ of
 $\XD$. 

Throughout this section, we denote by $y=(g_1, y^1, \dots, g_n, y^n)$
an element of $\ghat$.

\subsection{Functions associated to the center }\label{ss:centrefun}
\label{sec_hamcenter}

In this section, we introduce functions associated to the
center of the Levi factor $\fl_i$.

For each $i$, $1\le i\le n$, let $D_i$ be a decomposition class
of $\fg_i$. As in Remark \ref{rem_decomppartitition}, the decomposition
class $D_i$ determines a partition $i=i_1 + \dots + i_{r_i}$, 
normalized so $i_k \ge i_j$ when $k\le j$. Let $\fl_i$ be the standard
Levi subalgebra with diagonal blocks of sizes $(i_1, \dots, i_{r_i})$
and let $L_i\cdot u_i$ be the regular nilpotent orbit.
Then $\fz_i = \sum_{k=1}^{r_i} \C \cdot {\id}_{i_k}$, where
$ {\id}_{i_k}$ is the identity matrix in the kth block,
and $0$ outside the kth block. Consider the invariant symmetric
form on $\fg$ given by $< X, Y > = \Tr (X\cdot Y)$. The
restriction of  $< \cdot , \cdot >$ to $\fg_i$ is nondegenerate
and ${\fg_i}^\perp$ is spanned by elementary matrices
not in $\fg_i$. In particular, $\fg$ is a direct sum

\begin{equation}
\label{eq_perpfact}
\fg=\fg_{i}\oplus\fg_{i}^{\perp}.
\end{equation}
\begin{rem}\label{r:Ginvar}
It is easy to see that both components of the decomposition in (\ref{eq_perpfact}) are $\Ad(G_{i})$ and hence $\ad(\fg_{i})$ invariant.
\end{rem}

 Further, the restriction of the form
is nondegenerate on $\fl_i$ and $\fz_i$.
 Let $\{ z_{i,1}, \dots, z_{i, r_i} \}$ be the dual
basis in $\fz_i$ to the basis $\{ {\id}_{i_1}, \dots, {\id}_{i_{r_i}} \}$ of
$\fz_i$. Then $<z_{i,j}, z> = \lambda_j$ where
 $z= \sum_{i=1}^{r_{i}} \lambda_k {\id}_{i_k}$, so that pairing with $z_{i,j}$
computes the $jth$ eigenvalue of $z$. It is easy to check that

\begin{equation}
\label{eq_zijscalar}
z_{i,j} = \frac{{\id}_{i_j}}{i_j}.
\end{equation}

For $j=1, \dots, r_i$, we define a function
$\qij$ on $\gdotdtower = \prod_{k=1}^n G_k \times_{L_k} (\fz_{k,gen}
+ L_k \cdot u_k)$ by the formula:

\begin{equation}
\label{eq_semisimplefunctions}
\qij (g_1, y^1, \dots, g_n, y^n) =
 \sum_{s=i+1}^n <\Ad (g_s) y^s , \Ad (g_i) z_{i,j}>,
\end{equation}
where $g_k \in G_k$ and $y^k \in \fz_{k, gen} + L_k \cdot u_k$
for $k=1, \dots, n$. It is routine to check that
$\qij $ is a well-defined regular function on $\gdotdtower$.

For $i=1, \dots, n$,  we fix a parabolic subgroup $P_i$ containing $L_i$, and
recall that $\gtildedtower = \prod_{i=1}^n (G_i/P_i \times \fg_i).$

At a point $v=(g_1 P_1, x^1, \dots, g_n P_n, x^n)$ of $\gtildedtower$,
we identify the tangent space $T_v(\gtildedtower) = \oplus_{i=1}^n
(\fg_i/\fp_i + \fg_i)$ using the product of the identifications
of Equation (\ref{eq_tanidentify}). 
 Similarly, we identify
$T_v^*(\gtildedtower)= \oplus_{i=1}^n (\fu_i + \fg_i)$.
For each $j = 1, \dots, r_i$, we
define a covector $\tlambdaij \in \oplus_{i=1}^n (\fu_i + \fg_i)$ by
setting 

$$
\tlambdaij = (0,0, \dots, \underbrace{ 0, 0}_{\mbox{ith }}, 
\underbrace{ 0, \Ad (g_i) z_{i,j}}_{\mbox{i+1st }}, \dots,
\underbrace{ 0, \Ad (g_i) z_{i,j}}_{\mbox{nth }}),
$$




\noindent so that the first $i$ pairs are $(0,0) \in \fu_k \oplus \fg_k$,
$1\le k \le i$,
and the last $(n-i)$ pairs are $(0,  \Ad(g_{i}) z_{i,j}) \in
\fu_k \oplus \fg_k$, $i+1 \le k \le n$.

Recall the morphism $\alpha:\gdotdtower \to \gtildedtower$ from
Equation (\ref{eq_alphadef}).

\begin{prop}
\label{prop_tlambda}
Let $y\in \ghat$. Then 
$$
{\alphaD}_y^*(\tlambdaij) = d\qij (y).
$$
\end{prop}

\begin{proof} 
It suffices to verify that

\begin{equation}
\label{eq_chicheck}
\tlambdaij({\alphaD}_{y,*}(\chi))= d\qij (y)(\chi),
\end{equation}

\noindent for each tangent vector $\chi$ in a generating set of
tangent vectors  of $T_y(\gdotdtower)$.

For $B_k \in \fg_k$, we let $\xi_{B_k} \in T_y(\gdotdtower)$
 be the tangent
vector at $t=0$ to the curve
\begin{equation}
\label{eq_xiBk}
 t\mapsto (g_1, y^1, \dots, g_k \exp(tB_k), y^k, \dots, g_n, y^n),
\end{equation}

\par\noindent constant in all directions except the $(g_k, y^k)$ direction.
For $C_k \in \fz_k$, we let $\eta_{C_k} \in T_y(\gdotdtower)$
be the tangent vector at $t=0$ to the curve

\begin{equation}
\label{eq_etaCk}
t\mapsto (g_1, y^1, \dots, g_k, y^k + tC_k,  \dots, g_n, y^n),
\end{equation}

\par\noindent constant in all directions except the $(g_k, y^k)$ direction.
For $k=1, \dots, n$, the tangent vectors $\xi_{B_k}, \eta_{C_k}$
generate $T_y(\gdotdtower)$.

For $k < i$, $d\qij (\xi_{B_k})=0$ since $\qij$ is constant along
the corresponding flows. For $k \le i$, $d\qij (\eta_{C_k})=0$
for the same reason.

For the case $k=i$,

\begin{eqnarray}
d\qij (y)(\xi_{B_i}) & = &
 \ddt \sum_{s=i+1}^n  <\Ad(g_s) y^s , \Ad(g_i \exp(tB_i)) z_{i,j}> \nonumber\\
&= &\sum_{s=i+1}^n <\Ad(g_s) y^s, \Ad(g_i) [B_i, z_{i,j}]>\nonumber. 
\end{eqnarray}

We claim that $d\qij(y)(\xi_{B_i})=0$. For this, it suffices
to prove that for each $s > i$, $<\Ad(g_{s}) y^{s}, \Ad(g_{i})[B_{i},z_{i,j}]>=< [z_{i,j}, \Ad (g_i^{-1} g_s) y^s], B_i>=0$.
 By Equation (\ref{eq_perpfact}),

\begin{equation}
\label{eq_qijcompute}
< [z_{i,j}, \Ad (g_i^{-1} g_s) y^s], B_i> =
 <[z_{i,j}, \Ad (g_i^{-1} g_s ) y^s]_i, B_i >.
\end{equation}

\noindent Since $y\in \ghat$, $(\Ad(g_s)y_s)_i = \Ad(g_i)y_i$
 by Remark \ref{rem_ghatcriterion}.  Since
$z_{i,j} \in \fg_i$, it follows from Remark \ref{r:Ginvar} that we may rewrite (\ref{eq_qijcompute}) as

$$
<[z_{i,j}, \Ad (g_i^{-1}) (\Ad(g_{s}) y^s)_i], B_i>
= < [z_{i,j}, y^i], B_i> = 0,
$$

\par\noindent since $y^i \in \fl_i$.

For $k > i$, 
$$
d\qij (y)(\xi_{B_k}) = \ddt <\Ad (g_k \exp(tB_k)) y^k, \Ad (g_i) z_{i,j} >
= <\Ad (g_k) [B_k, y^k], \Ad (g_i) z_{i,j}>.
$$


\noindent It is now straightforward to verify Equation (\ref{eq_chicheck})
for $\chi = \xi_{B_k}$ using the formula for $\tlambdaij$.
 An easier version of the above computation
implies Equation (\ref{eq_chicheck}) for $\chi = \eta_{C_k}$.
This completes the proof.








\end{proof}

We now compute the Hamiltonian vector fields $\xi_{\qij}$
of the functions $\qij$ on $\ghat$.

\begin{prop}
\label{prop_hamfij}
For the Poisson structure $\piD$ on $\gdotdtower$ and a point
$y=(g_1,y^1, \dots, g_n, y^n)$ of  $\ghat$, the pushforward
$\alphaD_{y, *}\xi_{\qij}$ has $(\fg_k/\fp_k, \fg_k)$-component
$(0,0)$ for $k \le i$ and has $(\fg_k/\fp_k, \fg_k)$-component

$$
(-\Ad(g_{k}^{-1}g_i)z_{i,j} + \fp_{k}, [\Ad(g_{k}) y^{k},\Ad(g_i)z_{i,j}])
$$

\noindent for $k > i$.

\end{prop}

\begin{proof}
By Proposition \ref{prop_tlambda}, 

$$
\alphaD_{y, *}(\xi_{\qij})=\alphaD_{y, *} \widetilde{\piD}_{y} \alphaD_y^*(\tlambdaij).$$
Since the embedding $\alphaD : (\gdotdtower, \piD) \to (\gtildedtower, \PiD)$
is Poisson (see Section \ref{sec_poissontildevariety}), 

$$
\alphaD_{y, *}  \widetilde{\piD}_{y} 
{\alphaD}_y^* (\tlambdaij) = \widetilde{\PiD}_{ \alphaD(y)}(\tlambdaij).
$$
The Proposition now follows easily from 
Equation (\ref{eq_anchorboth}).
\end{proof}

We now define a curve $\theta_{i,j}(y,t)$ that integrates the vector
field $\xi_{\qij}$ on $\ghat$. For $y=(g_1,y^1, \dots, g_n, y^n) \in \gdotdtower$,
we let 
$$
h_{i,j}(y, t)= \exp(-t \Ad(g_i)z_{i,j}).
$$

We define the curve $\theta_{i,j}(y,t)$ in $\gdotdtower$ by the equation
\begin{equation}
\label{eq:centrecurve}
t\mapsto (g_1, y^1, \dots, g_i, y^i, h_{i,j}(y, t) g_{i+1}, y^{i+1}, \dots,
h_{i,j}(y, t) g_n, y^n).
\end{equation}

\begin{prop}
\label{prop_algintegralfij}
The curve $\theta_{i,j}(y,t)$ is an integral curve for $\xi_{\qij}$ on
$\ghat$, and induces an algebraic action of $\Co^{\times}$ on $\ghat$.
\end{prop}

\begin{proof}
Suppose $y\in\ghat $.  Let $x^i = \Ad(g_i)y^i$ for $i=1, \dots, n$.
To show that $\theta_{i,j}(y,t) \in \ghat$,  we must verify that for $k > i$,
\begin{equation}
\label{eq:cutoffs}
(\Ad(h_{i,j}(y, t)) x^n)_{k}=\Ad(h_{i,j}(y, t))x^k
\end{equation}
 and for $k \le i$,
\begin{equation}\label{eq:lowcutoffs}
(\Ad(h_{i,j}(y, t)) x^n)_{k}=x^k,
\end{equation}
  and also verify that $\Ad(h_{i,j}(y,t)) x^{n}$ is strongly regular.  
For $k > i$, $h_{i,j}(y, t)\in G_{i}\subset G_{k}$,  so by Remark \ref{r:Ginvar} we may rewrite the left side
 of (\ref{eq:cutoffs}) as
$$
\Ad(h_{i,j}(y, t)) (x^n)_{k}=\Ad(h_{i,j}(y, t)) x^k,
$$  
since $(x^n)_k=x^k$ by Remark \ref{rem_ghatcriterion}.  
Now, if $k\leq i$, then
 $$
(\Ad(h_{i,j}(y, t)) x^n)_{k}=(\Ad(h_{i,j}(y, t)) (x^n)_i)_{k}=
(\Ad(h_{i,j}(y, t)) x^i)_{k}=(x^i)_{k}=x^k,
$$
 since $h_{i,j}(y, t)$ centralizes $x^i$.  To verify
 that $\Ad(h_{i,j}(y,t)) x^{n}$ is strongly regular,
we use Equations (\ref{eq:cutoffs}) and (\ref{eq:lowcutoffs}).  We first observe that $(\Ad(h_{i,j}(y,t))x^{i})_{k}$ is regular in $\fg_{k}$ by Equations (\ref{eq:cutoffs}) and (\ref{eq:lowcutoffs}) and Theorem \ref{thm_srchar} (3).  
For $k<i$, Equation (\ref{eq:lowcutoffs}) implies,

  $$\fz_{\fg_{k}}((\Ad(h_{i,j}(y,t))\cdot 
x^{n})_{k})\cap \fz_{\fg_{k+1}}((\Ad(h_{i,j}(y,t))\cdot x^{n})_{k+1})=
\fz_{\fg_{k}}(x^{k})\cap\fz_{\fg_{k+1}}(x^{k+1})=0,$$ by Theorem \ref{thm_srchar} (3). 
 For $k\geq i$, by (\ref{eq:cutoffs}),
 \begin{gather*}
 \fz_{\fg_{k}}((\Ad(h_{i,j}(y,t))\cdot x^{n})_{k})\cap 
\fz_{\fg_{k+1}}((\Ad(h_{i,j}(y,t))\cdot x^{n})_{k+1}) \\
= \Ad(h_{i,j}(y,t))(\fz_{\fg_{k}}(x^{k})\cap\fz_{\fg_{k+1}}(x^{k+1}))=0.
\end{gather*}
 Thus, by Theorem \ref{thm_srchar},  
$\Ad(h_{i,j}(y,t))x^{n}$ is strongly regular and  
 $\theta_{i,j}(y,t) \in \ghat$ for $y\in\ghat$.
 
By Proposition \ref{prop_hamfij}, it follows that 
$\frac{d}{dt} \theta_{i,j}(y,t)
=\xi_{\qij}$.  Consider the element $\tilde{h}_{i,j}(y,s)=g_{i}A_{j}(s)g_{i}^{-1}\in G_{i}$, where $s\in\Co^{\times}$ and $A_j(s)$ is the diagonal matrix in
$G_i$ with scalar matrix $s$ in the jth block of $\fl_i$ and $1$ elsewhere on the diagonal.  We define a curve ${\tilde{\theta}}_{i,j}(y,s)$ in $\gdotdtower$ 
using Equation (\ref{eq:centrecurve}) with 
$\tilde{h}_{i,j}(y,s)$ in place of $h_{i,j}(y,t)$.  
It follows from Equations (\ref{eq_zijscalar}) and (\ref{eq:centrecurve})
 that ${\tilde{\theta}}_{i,j}(y,s) =
\theta_{i,j}(y,t)$ when $s=\exp(-\frac{t}{i_j})$, which completes the
proof of the proposition.

\end{proof}

\subsection{Functions associated to the semisimple part }\label{ss:nilfun}
\label{sec_hamsemisimple}

Let $s_i = i - \dim(\fz_i)$. In this section, we find $s_i$ algebraically independent
 functions on $\gdotdtower$ and
show their Hamiltonian vector fields integrate to an
algebraic action on $\ghat$.

Let $\fl_i$ be the standard Levi factor associated to the
decomposition class $D_i$. We decompose $\fl_i = \fs_i \oplus \fz_i$,
where $\fs_i=[\fl_i,\fl_i]$ is the derived algebra of $\fl_i$, and note
that the form $\lara$ is nondegenerate on $\fs_i$.

For $x\in \fs_i$, and $f\in \C [\fs_i]$, we define the
gradient $\nablaf (x) \in \fs_i$ by the property that
$< \nablaf (x), u > = df(x)(u)= \ddt f(x+tu)$, for
$u\in \fs_i$. It follows from definitions that
if $g\in L_i$ and $\Ad(g)f=f$, then 

\begin{equation}
\label{eq_nablainv}
\nablaf (\Ad(g) x) =
\Ad(g) (\nablaf (x)).
\end{equation}

Further, if $f \in \C[\fs_i]^{Z_{L_i}(x)}$, then

\begin{equation}
\label{eq_nablacentral}
[\nablaf (x), u] = 0 \ \forall \ u \in {\fz}_{\fs_{i}}(x).
\end{equation}

\begin{lem}
\label{lem_gradientprop}
Let $f_1, f_2 \in \C [\fs_i]^{L_i}$.
Then $[\nabla f_{1}(x), \nabla f_{2}(x)]=0$.
\end{lem}

\begin{proof} 
By Equation (\ref{eq_nablacentral}) with $u=x$, $\nabla f_{i}(x) \in \fz_{\fs_i}(x)$.  The result follows immediately from another application of Equation (\ref{eq_nablacentral}) with $u=\nabla f_{i}(x)$.  
\end{proof}

Let now $f\in \C [\fs_i]^{L_i}$ be an invariant function,
and define a regular function $\tilde{f}$ on $\gdotdtower$ by
the formula:

\begin{equation}
\label{eq_fi}
\tilde{f} (g_1, y^1, \dots, g_n, y^n) = \sum_{s=i+1}^n <\Ad ( g_s) y^s,
\Ad (g_i) \nablaf (n^i)>,
\end{equation}

\noindent where $y^i = s^i + n^i$ is the Jordan decomposition of $y^i$.
By Equation (\ref{eq_nablainv}),
$\tilde{f}$ is a well-defined regular function on $\gdotdtower$.

Let $y=(g_1, y^1, \dots, g_n, y^n) \in \ghat$.
 Using the
identification $T^*_{\alphaD (y)}(\gtildedtower) = \oplus (\fu_i + \fg_i)$,
we regard the vector

$$
\psi = (0, 0, \dots, \underbrace{0, 0}_ {\mbox{ith }}, 
\underbrace{ 0, \Ad(g_i) \nablaf (n_i)}_{\mbox{i+1st }}, \dots,\underbrace{ 0, \Ad(g_i) \nablaf (n_i)}_{\mbox{nth }})
$$

\par\noindent as a cotangent vector at $\alphaD(y) \in \gtildedtower$. 

\begin{prop}
\label{prop_dfi}
For $y\in \ghat$, ${\alphaD}_y^* (\psi) = d\tilde{f} (y).$
\end{prop}

\begin{proof} We must check that 

\begin{equation}
\label{eq_propdfivec}
\psi(\alphaD_{*,y}(\chi)) =
d\tilde{f} (y)(\chi)
\end{equation}
 for each tangent vector $\chi$ in a generating set
of tangent vectors of $T_y (\gdotdtower)$. It suffices to check
this assertion when $\chi$ is the evaluation of a vector field $\xi_{B_k}$
with $B_k \in \fg_k$ or the evaluation of a vector field $\eta_{C_k}$
with $C_k \in \fz_k$ at $y$  (see Equations (\ref{eq_xiBk}) and (\ref{eq_etaCk})).
 If $k < i$, the assertion is trivial as both sides are $0$.
If $k > i$, we compute

$$
d\tilde{f} (y)(\xi_{B_k}) = < \Ad(g_k) [B_k, y^k], \Ad(g_{i})\nablaf (n^i) >
= \psi (\alphaD_{*, y} \xi_{B_k}). 
$$

If $k=i$, we compute 

\begin{equation}
\label{eq_dfifirst}
d\tilde{f} (y)(\xi_{B_i}) = \ddt \sum_{s=i+1}^n < \Ad(g_{s})y^{s}, \Ad(g_{i}\exp(tB_{i}))\nablaf(n^{i})> .
\end{equation}

\noindent Since $\exp(-tB_i) \in G_i$, $\nablaf (n^i) \in \fg_i$,
and $(\Ad(g_s) y^s)_i = \Ad (g_i) y^i$, 


\begin{eqnarray}
\label{eq_dfisecond}
\lefteqn{< \Ad(g_{s})y^{s}, \Ad(g_{i}\exp(tB_{i}))\nablaf(n^{i})>=< \Ad (\exp (-tB_i) g_i^{-1} g_s)
y^s , \nablaf (n^i) >} \\
\nonumber\\ 
 & & = < (\Ad (\exp (-tB_i) g_i^{-1} g_s) y^s)_i, \nablaf (n^i) >
 = < \Ad (\exp (-tB_i))y^i, \nablaf (n^i) >\nonumber.
\end{eqnarray}




Hence, we may rewrite Equation (\ref{eq_dfifirst}) as

\begin{equation}
\label{eq_dfithird}
d\tilde{f} (y)(\xi_{B_i}) = \ddt \sum_{s=i+1}^n < \Ad (\exp (-tB_i)) y^i,
 \nablaf (n^i) > .
\end{equation}

But

\begin{equation}
\label{eq_dfifouth}
\ddt < \Ad (\exp (-tB_i)) y^i, \nablaf (n^i) > = < -B_i, [y^i, \nablaf (n^i)]>
= 0
\end{equation}

\par\noindent using Equation (\ref{eq_nablacentral}), 
and Equation (\ref{eq_propdfivec}) follows easily for $\chi = \xi_{B_k}$.
The verification of Equation (\ref{eq_propdfivec})
for $\chi = \eta_{C_k}$ is left to the reader.
\end{proof}

We denote by $y^i = s^i + n^i$ the Jordan decomposition
of $y^i \in D(\fl_i, u_i)$ for the remainder of the paper.
We choose $s_i$ algebraically independent functions
$\phi_{i,j} \in \C [\fs_i]^{L_i}$. For $j=1, \dots, s_i$,
we define a regular function $p_{i,j}$ on
$\gdotdtower$ by $p_{i,j} = {\tilde{\phi}}_{i,j}$, using
Equation (\ref{eq_fi}) with $f=\phi_{i,j}$.

\begin{rem}
\label{rem_basiszgini}
The set $\{ \nabla \phi_{i,j}(n^i) : j=1, \dots, s_i \}$ is a basis of
$\fz_{\fs_i}(n^i)$.  First note that $\nabla \phi_{i,j}(n^i)\in\fz_{\fs_{i}}(n^{i})$ by Equation (\ref{eq_nablacentral}).  Since $n^i$ is principal nilpotent in $\fs_i$, $\dim\fz_{\fs_{i}}(n^{i})=\text{rank}(\fs_{i})=s_{i}$, and the elements $\{ \nabla \phi_{i,j}(n^i) : j=1, \dots, s_i \}$ are linearly independent by a well-known result of Kostant \cite{K}, Theorem 9.
\end{rem}

By Proposition \ref{prop_dfi}, it follows that for $y \in \ghat$,

\begin{equation}
\label{eq_dpij}
dp_{i,j}(y) = {\alphaD}_y^* (\psi_{i,j}),
\end{equation}

\noindent where $\psi_{i,j}$ is the vector in $\oplus (\fu_i + \fg_i)$
given by

\begin{equation}
\label{eq_psiij}
\psi_{i,j}=(0,0,\dots, \underbrace{0, 0}_{\mbox{ith }},
\underbrace{0, \Ad(g_i)\nabla \phi_{i,j}(n^i)}_{\mbox{i+1st}}, \dots, 
\underbrace{ 0, \Ad(g_i)\nabla \phi_{i,j}(n^i)}_{\mbox{nth}}).
\end{equation}

\begin{prop}
\label{prop_hampij}
For the Poisson structure $\piD$ on $\gdotdtower$ and 
$y=(g_1,y^1, \dots, g_n, y^n) \in  \ghat$,
 the $kth$ component in $(\fg_k/\fp_k, \fg_k)$ of the pushforward $\alpha_{y,*}\xi_{p_{i,j}}$ is $(0, 0)$ if $k\le i$,
and if $k > i$, it is 
$$
(-\Ad(g_{k}^{-1}g_i)\nabla \phi_{i,j} (n^i) + \fp_{k}, [\Ad(g_{k}) y^k, \Ad(g_i)\nabla \phi_{i,j} (n^i)]).
$$
\end{prop}

\noindent The proof is similar to the proof of Proposition \ref{prop_hamfij}, and
we leave details to the reader.

 We define

\begin{equation}
\label{eq_Aijy}
A_{i,j}(y,t) = \exp (-t \Ad(g_i)(\nabla \phi_{i,j} (n^i))).
\end{equation}

Since $n^i$ is principal nilpotent in $\fs_i$,
$\nabla \phi_{i,j} (n^i)\in\fz_{\fs_{i}}(n^{i})$ is nilpotent by a well-known result
of Kostant (see Remark 35.1.4 in \cite{TY}). It follows that
 the morphism
$t\mapsto A_{i,j}(y,t)$ is algebraic.

For each $j=1, \dots, s_i$, define a curve in $\gdotdtower$ by the formula

\begin{equation}
\label{eq_muij}
\mu_{i,j}(y,t) = 
(g_1, y^1, \dots, g_i, y^i, A_{i,j}(y,t) g_{i+1}, y^{i+1},
\dots, A_{i,j}(y,t) g_n, y^n).
\end{equation}

\begin{prop}
\label{prop_algintpij}
The curve $\mu_{i,j}(y,t)$ is an integral curve for
the Hamiltonian vector field $\xi_{p_{i,j}}$ on $\ghat$,
 and induces an algebraic action of the
additive group $\C$ on $\ghat$.
\end{prop}

The proof is similar to the proof of Proposition \ref{prop_algintegralfij} and the details are left to the reader.  

\subsection{Lift of the $A$-action to $\ghat$}
\label{sec_lift}

Let 

$$
\hat{J_i} = \{ \qij :  j=1, \dots, r_i=\dim (\fz_i) \}
\cup \{ p_{i,k} : k=1, \dots, s_i = i - r_i \},
$$

\noindent and let $\hat{J} =\bigcup_{i=1}^{n-1} \hat{J_i}$.

Let 
\begin{equation}\label{eq:hata}
{\hfa}_i :=  \lspan \{ \xi_f : f\in \hat{J_i} \} \subset 
\Gamma (\ghat, T\ghat), \ \
\hfa := \sum_{i=1}^{n-1} {\hfa}_i.
\end{equation}


Recall the {\'e}tale covering $\mu:\ghat \to \XD$.
In this section, we show that $\hat{\fa}$ is an abelian Lie algebra of 
dimension $\dn$, and $\mu_* \hat{\fa} = \fa$, so that $\hat{\fa}$
  lifts the action of
$\fa$ to the covering $\ghat$.  

\begin{prop}
\label{prop_ahatabelian}

If $\xi_f$ and $\xi_g \in \hat{\fa}$, then $[\xi_f, \xi_g](y)=0$
for all $y \in \ghat$.
\end{prop}

\begin{proof}
To prove the Proposition, it suffices to show that the flows corresponding
to generating vector fields in $\hat{\fa}$ commute, i.e., for $i\le k$,

\begin{equation}
\label{eq_commutingflows}
\begin{split}
\theta_{i,j}(\theta_{k,l}(y,s), t) &= 
\theta_{k, l}(\theta_{i,j}(y,t),s), \ \forall \ y\in \ghat \\
\mu_{i,j}(\mu_{k,l}(y,s), t) &= 
\mu_{k, l}(\mu_{i,j}(y,t),s), \ \forall \ y\in \ghat \\
\theta_{i,j}(\mu_{k,l}(y,s), t) &= 
\mu_{k, l}(\theta_{i,j}(y,t),s), \ \forall \ y\in \ghat .
\end{split}
\end{equation}

For this, let $u\in G_i$ and $v\in G_k$.  For a point 
$q=(g_1, y^1, \dots, g_n, y^n) \in \Pi_{a=1}^n G_a \times \fg_a$, 
define $h(q)=g_i u g_i^{-1}$ and $l(q)=g_k v g_k^{-1}$. Set

$$
r_i(q)=(g_1, y^1, \dots, g_i, y^i, h(q)g_{i+1}, y^{i+1}, \dots, h(q)g_n, y^n),
$$

$$
w_k(q)=(g_1, y^1, \dots, g_k, y^k, l(q)g_{k+1}, y^{k+1}, \dots, l(q)g_n, y^n).
$$

The identities in Equation (\ref{eq_commutingflows}) reduce to the equation

\begin{equation}
\label{eq_generalcommute}
r_i (w_k(q)) = w_k (r_i(q))
\end{equation}
for particular choices of $u$ and $v$.

We first assume $i < k$. Then if $a \le k$, it is routine to check that
the $G_a \times \fg_a$ coordinates in Equation (\ref{eq_generalcommute})
coincide. For $a > k$, the $G_a \times \fg_a$-coordinate of $r_i(w_k(q))$ is $(h(q)l(q)g_a, y^a)$, while 
the $G_a \times \fg_a$-coordinate
of $w_k(r_i(q))$ is $(h(q)l(q)h(q)^{-1} h(q) g_a, y^a)$, so Equation
(\ref{eq_generalcommute}) is easily verified.

In the case $i=k$, note that
 Equation (\ref{eq_generalcommute}) is easy to verify when $uv=vu$. Thus, to verify
Equations (\ref{eq_commutingflows}), it suffices to check that
$uv=vu$ when
$u$ and $v$ are chosen from $\exp(-tz_{i,j})$ and $\exp(-s\nabla \phi_{i,l} (n^i))$,
for $j=1, \dots, r_i$ and $l=1, \dots, s_i$, and for any $t,\, s\in\Co$.  This follows by Lemma \ref{lem_gradientprop}.
\end{proof}

We now show that $\hfa$ lifts the action of $\fa$ to $\ghat$.

\begin{lem}\label{l:ahat}
Let $\mu: \ghat\to\XD$ be the \'{e}tale covering defined in Section \ref{sec_ghat}.
  Then $\mu_{*}\hfa=\fa$, so that $\hfa$ lifts the action of $\fa$ to $\ghat$. 
\end{lem}

\begin{proof}
Consider the commutative diagram
\begin{equation}\label{eq:twosquares}
\begin{array}{ccccc}
\ghat &\to & \gdotdtower&\stackrel{\alphaD}{\to}&\gtildedtower \\
\downarrow{\mu}  &   & \downarrow{\mu}&   &\downarrow{\Psi} \\
\XD & \stackrel{\gamma}{\to} & \Pi_{i=1}^n D_i&\to&  \Pi_{i=1}^{n} D_{i}
\end{array}_{\mbox{,}}
\end{equation}
where the first square is Diagram (\ref{eq_ghatcartesian}) and the map
 $\Psi:\gtildedtower\to   \Pi_{i=1}^{n} D_{i}$ is projection on
$\Pi_{i=1}^n D_i$.
  For $z\in \Pi_{i=1}^{n} D_{i}$ we identify $T_{z}( \Pi_{i=1}^{n} D_{i})$ as
 a subspace of $T_{z}(\Pi_{i=1}^{n} \fg_{i})=\oplus_{i=1}^{n} \fg_{i}$.  
Let $y = (g_{1}, y^{1},  \dots, g_{n}, y^{n}) \in\ghat$.
  Then $\alphaD(y)=(g_{1}P_{1}, x^{1}, \dots, g_{n} P_{n}, x^{n})$,
 where $x^{i}=\Ad(g_{i}) y^{i}$.  We let $x=x^{n}$, so $x^{i}=x_{i}$ by Remark
\ref{rem_ghatcriterion}. 
By Propositions \ref{prop_hamfij} and \ref{prop_hampij},  

\begin{equation}\label{eq:centham}
\Psi_{ \alphaD(y),*}\alphaD_{ y,*}\xi_{q_{i,j}}=(0,0, \dots, 
\underbrace{0}_{\mbox{ith}},\underbrace{ [x_{i+1}, \Ad(g_{i}) 
z_{i,j}]}_{\mbox{i+1st}},\dots, \underbrace{[x, \Ad(g_{i}) z_{i,j}]}_{\mbox{nth}}),
\end{equation}
for $1\leq i\leq n-1$, $1\leq j\leq r_{i}$, and 
\begin{equation}\label{eq:nilham}
\Psi_{ \alphaD(y),*}\alphaD_{ y,*}\xi_{p_{i,j}}=(0,0, \dots, 
\underbrace{0}_{\mbox{ith}}, 
\underbrace{[x_{i+1}, \Ad(g_{i})\nabla\phi_{i,j}(n^{i})]}_{\mbox{i+1st}},\dots, 
\underbrace{ [x, \Ad(g_{i}) \nabla\phi_{i,j}(n^{i})]}_{\mbox{nth}}),
\end{equation}
for $1\leq i\leq n-1$, $1\leq j\leq s_{i}$.

Recall that $\xifij\in\fa$, where $f_{i,j}(x)=tr(x_{i}^{j})$.  
Then identifying $T_{z}(\fg) = \fg$ for $z\in \fg$, it follows from 
 Theorem 2.12 in \cite{KW1} that $(\xifij)_{z}=[jz_{i}^{j-1}, z]$.  
Hence,

$$
\gamma_{x,*}\xifij=([jx_{i}^{j-1},x]_{1}, [jx_{i}^{j-1},x]_{2}, \dots, [jx_{i}^{j-1}, x]).
$$

From Remark \ref{r:Ginvar}, it follows that for $k>i$, $[jx_{i}^{j-1}, x]_{k}=[jx_{i}^{j-1}, x_{k}]$ and for $k\leq i$, $[jx_{i}^{j-1}, x]_{k}=0$.  
%
  Thus,
 
\begin{equation}\label{eq:push}
\gamma_{x,*}\xifij=(0,0\dots, \underbrace{0}_{\mbox{ith}},
\underbrace{[jx_{i}^{j-1},x_{i+1}]}_{\mbox{i+1st}},\dots, [jx_{i}^{j-1},x]).
\end{equation}

Since $y\in\ghat$, $x\in\fg_{sreg}$ by Remark \ref{rem_ghatcriterion}, 
so that by Theorem \ref{thm_srchar}, $x_{i}$ is regular for all $i$. 
Thus, 

$$\lspan\{x_{i}^{j-1}:\; 1\leq j\leq i\}=\fz_{\fg_{i}}(x_{i}),$$
by a standard result from linear algebra.  We claim 

$$
\lspan\{\Ad(g_{i}) z_{i,j}, \, \Ad(g_{i}) \nabla\phi_{i,k}(n^{i}) :\; 1\leq j\leq r_{i},\, 1\leq k\leq s_{i}\}=\fz_{\fg_{i}}(x_{i}).
$$
Indeed, $\Ad(g_{i}^{-1})x_{i}=y^{i}=s^{i}+n^{i}$ with $n^{i}\in\fs_{i}$ 
principal nilpotent.  By Remark \ref{rem_basiszgini},
$\{\nabla \phi_{i,k} (n^{i}) :\, 1\leq k\leq s_{i}\}$ is a basis 
of $\fz_{\fs_{i}}(n^{i})$.  
Since $\fz_{\fg_{i}}(y^{i})=\fz_{i}\oplus \fz_{\fs_{i}}(n^{i}),$
the claim follows easily.

By Equations (\ref{eq:centham}), (\ref{eq:nilham}), and (\ref{eq:push}),
$\lspan\{\gamma_{x,*}\xifij: \; 1\leq j\leq i-1\}$ equals
\begin{equation}
\lspan\{ \Psi_{\alphaD(y),*} \alphaD_{y,*}\xi_{q_{i,j}}: \; 1\leq j\leq r_{i}\}\oplus \lspan\{ \Psi_{\alphaD(y),*} \alphaD_{y,*}\xi_{p_{i,k}}: \; 1\leq k\leq s_{i}\}
\end{equation}
for each $i$, $1\leq i\leq n-1$ and any $y\in\ghat$.  
Thus, $\gamma_{*}\fa=\Psi_{*}\alpha_{*}\hfa=\gamma_{*}\mu_{*}\hfa$, so
since $\gamma$ is an embedding, $\fa=\mu_{*}\hfa$.

\end{proof}

\subsection{The integration of the $\hfa$-action on $\ghat$}
\label{sec_integration}

In this section, we show that the Lie algebra $\hfa$ integrates to an
algebraic action of a connected abelian algebraic group on $\ghat$.

As before, $\Dtower = (D_1, \dots, D_n)$ is regular decomposition
data with $D_i = G_i \cdot (\fz_{i, gen} + u_i)$.
Recall the identification from Section \ref{sec_ghat},

$$
\ghat \cong \Dinc = \{ (x, z_1, \dots, z_n) \in \fg_{sreg} \times 
\fzD : x_i \in G_i \cdot (z_i + u_i) \},
$$ 

\noindent and the projections $\mu:\Dinc \to \fg_{sreg}$ and $\kappa: \Dinc \to \fzD.$
For $(z_1, \dots, z_n) \in \fzD$, 
\begin{equation}
\label{eq:kappafibre}
\kappa^{-1}(z_1, \dots, z_n)\stackrel{ \mu}{\cong} \{ x\in \XD : x_i \in G_i \cdot
(z_i +u_{i}), i=1, \dots, n \} =
\Phi^{-1}(\Phi_1(z_1), \dots, \Phi_n(z_n))\cap\fg_{sreg}
\end{equation}
by Lemma \ref{lem_orbitid}.

It is convenient to normalize the nilpotent matrix $u_i$ so that $u_i=e^i$
is the unique principal nilpotent element of $\fl_i$ in Jordan canonical
form. Let $Z_{D_i} = Z_{G_i}(\fz_i + e^i)=Z_{L_i}(e_i)$.  
Note that $Z_{D_{i}}$ is connected and abelian because it is the centralizer of a 
regular element of $\fl_{i}$ (see Proposition 14 in \cite{K}).  We identify
$$
\coverD (\fl_i, e^i) = G_i \times_{L_i} (\fz_{i, gen} + L_i \cdot e^i)
\cong G_i \times_{Z_{D_i}} (\fz_{i,gen} + e^i),
$$
and identify $\gdotdtower \cong \Pi_{i=1}^n G_i \times_{Z_{D_i}} (\fz_{i,gen} + e^i).$
The abelian algebraic group $Z_{D_i}$ has Levi decomposition
$Z_{D_i} = Z_i \times (Z_{D_i})_u$ with unipotent radical
$(Z_{D_i})_u$, and its Lie algebra $\fz_{D_i} = \fz_i \oplus \fz_{\fs_i}(e^i)$.
 Recall that $\exp:\fz_{\fs_i}(e^i) \to (Z_{D_i})_u$
is an isomorphism. Further $\fz_{\fs_i}(e^i)$ has basis given by the
elements  $\nabla \phi_{i,j}(e^i)$, for $ j=1, \dots, s_i$ by
Remark \ref{rem_basiszgini}.

\begin{lem}
\label{l:trivialint}
If $x\in\fg_{sreg}$, then 
$Z_{G_{i}}(x_{i})\cap Z_{G_{i+1}}(x_{i+1})=\{ e \}$ 
for all $1\leq i\leq n-1$, where $e\in G_{n}$ denotes the identity matrix.  
\end{lem}

\begin{proof}
If $A\in Z_{G_{i}}(x_{i})\cap Z_{G_{i+1}}(x_{i+1})$, then 
it follows easily that $A - e \in \fz_{\fg_i}(x_i)
 \cap \fz_{\fg_{i+1}}(x_{i+1}) = 0$ by Theorem \ref{thm_srchar},
which implies the lemma.
\end{proof}

We consider the connected, abelian algebraic group $\ZD = Z_{D_1} \times \dots 
\times Z_{D_{n-1}}$.

\begin{thm}\label{thm:ZDacts}
(1) The Lie algebra $\hfa$ integrates to a free algebraic action 
of the group $\ZD$ on $\ghat$.  This action of $\ZD$ preserves
 the fibers $\kappa^{-1}(\uz)$ for $(\uz)\in\fzD$. 

\noindent (2) The orbits of $\ZD$ in $\kfibre$ are the irreducible
components of $\kfibre$. If we let $j_i$ denote the cardinality of the set 
$\sigma_{i}(z_{i})\cap \sigma_{i+1}(z_{i+1})$ for $1\leq i\leq n-1$,
  then the number of $\ZD$-orbits in $\kfibre$ is exactly 
$2^{\sum_{i=1}^{n-1} j_{i}}$.
\end{thm}

\begin{proof}
Let $k_i \in Z_{D_i}$ and let $y=(g_1, y^1, \dots, g_n, y^n) \in \ghat$
with $y^i \in \fz_{i,gen} + e^i$.
The group $Z_{D_i}$ acts on $\ghat$ by the formula

\begin{equation}
\label{eq_zdiaction}
k_i\cdot y =  (g_1, y^1, \dots, g_i, y^i, g_i k_i g_i^{-1}g_{i+1}, y^{i+1},
g_i k_i g_i^{-1} g_n, y^n).
\end{equation}

We claim that $T(Z_{D_i} \cdot y) = {\hfa}_i$ (see (\ref{eq:hata})). Indeed,
$T_y(Z_{D_i} \cdot y) = T_y(Z_i \cdot y) + T_y((Z_{D_i})_u \cdot y)$.
By Equations (\ref{eq:centrecurve}) and (\ref{eq_muij}), for
$y\in \ghat$,
 $T_y(Z_{i} \cdot y) =
\lspan \{ \xi_{\qij}(y) : j=1, \dots, r_i \}$ and $T_y((Z_{D_i})_u \cdot y)
= \lspan \{ \xi_{p_{i,k}}(y) : k=1, \dots, s_i \}$, which gives
the claim. Hence, 
 ${\hfa}_i$ induces the tangent space to the $Z_{D_i}$-action
at every point $y$ of $\ghat$, so the action of ${\hfa}_i$ integrates to the algebraic
action of the algebraic group $Z_{D_i}$ on $\ghat$.

Using Equation (\ref{eq_commutingflows}),
it is easy to verify that 
 the actions of $Z_{D_i}$ and $Z_{D_k}$ commute
for $i\not= k$. Hence, $\ZD$ acts on $\ghat$ by the formula

\begin{equation}
\label{eq_zdaction}
(k_1, \dots, k_{n-1}) \cdot y = k_1\cdots k_{n-1} \cdot y, \, k_i \in Z_{D_i}, \, y\in \ghat.
\end{equation}

\noindent or more explicitly,

\begin{eqnarray}
\label{eq:bigact}
 & \lefteqn{(k_{1},\dots, k_{n-1})\cdot y = } \\ & (g_{1},y^{1}, g_{1}k_{1}g_{1}^{-1}g_{2}, y^{2},\dots,\underbrace{ g_{1}k_{1}g_{1}^{-1}\dots g_{i-1}k_{i-1}g_{i-1}^{-1} g_{i}, y^{i}}_{\mbox{ith}},\dots, \underbrace{ g_{1}k_{1}g_{1}^{-1}\dots g_{n-1}k_{n-1}g_{n-1}^{-1}g_{n}, y^{n}}_{\mbox{nth}}).\nonumber
\end{eqnarray}

Since ${\hfa}_{i}$ integrates to an algebraic action of $Z_{D_{i}}$ for each $i$, $1\leq i\leq n-1$, and the actions of the groups $Z_{D_{i}}$ commute, it follows that $\hfa$ integrates to an algebraic action of $\ZD$ on $\ghat$.  By Equations (\ref{eq_zdiaction}) and (\ref{eq_zdaction}), the $\ZD$-action
on $\ghat$  preserves the fibers $\kappa^{-1}(\uz)$ for $(\uz)\in\fzD$. 



To prove that the $\ZD$-action on $\ghat$ is free, we show by induction
that if $k=(k_1, \dots, k_{n-1})$ fixes $y \in \ghat$, and $k_1, \dots,
k_{i-1}=e$, then $k_i = e$. Indeed, then $k\cdot y$ has $G_{i+1} 
\times \fg_{i+1}$ coordinate $(g_i k_i g_i^{-1} g_{i+1}, y^{i+1})$,
so by hypothesis,  $\Ad(g_i k_i g_i^{-1} g_{i+1})
y^{i+1} = \Ad(g_{i+1}) y^{i+1}$. Thus, if we set $x^i = \Ad(g_i)y^i$
and $x^{i+1} = \Ad(g_{i+1})y^{i+1}$, then
 $g_i k_i g_i^{-1} \in Z_{G_i}(x^i) \cap Z_{G_{i+1}}(x^{i+1}) = \{ e \}$
by Lemma \ref{l:trivialint}. Hence, $k_i = e$, which establishes the
inductive step, so the $\ZD$-action is free.

To prove (2), we first observe that since the connected 
algebraic group $\ZD$ acts freely 
on $\ghat$, each $\ZD$-orbit in $\kfibre$ is an irreducible subvariety of $\kfibre$ of 
dimension $\dn$.  It follows that all orbits of $\ZD$ on $\kfibre$ are closed.
Let $c = (\Phi_1(z_1), \dots, \Phi_n(z_n))$. 
Using the isomorphism $\mu:\kfibre \to \Phi^{-1}(c)_{sreg}$ of 
Equation (\ref{eq:kappafibre}),
it follows that if $x\in \kfibre$, then $\mu(\ZD \cdot x)$ is a 
closed, irreducible subvariety of dimension $\dn$ of $\Phi^{-1}(c)_{sreg}$.
By Theorem 3.12 in \cite{KW1}, each irreducible component of $\Phi^{-1}(c)_{sreg}$ is
an $A$-orbit of dimension $\dn$, which implies the first statement of (2).
The last statement of (2) now follows from Theorem \ref{thm:bigthm}.

 
 
 \end{proof}

\begin{rem}\label{r:newerrem}
Let $\mathcal{D}=(D_{1},\dots, D_{n})$ with $D_{i}=D(\fl_{i},u_{i})$ and suppose for each $i$ that all blocks of $\fl_{i}$ have different sizes (see Remarks \ref{rem_decomppartitition}, \ref{r:newrem}).  In this case, Theorem \ref{thm:ZDacts} implies that $\fa$ integrates to an algebraic action of $\ZD$ on $\XD$.  
\end{rem}

\begin{rem}
\label{rem_relcol2}
Let $x\in \XD$ and let $\Phi(x)=c \in \C^{\dnone}$. 
In Section 4 of \cite{Col1}, the first author constructed an algebraic $\ZD$-action on $\Phi^{-1}(c)_{sreg}$ (see Theorem \ref{thm:bigthm}). For $y\in \ghat$ 
such that $\mu(y)\in \Phi^{-1}(c)_{sreg}$ and $z\in \ZD$, $\mu(z\cdot y)=z\cdot \mu(y)$.
 This can
be verified using the formula for the action  in \cite{Col1}.
However, the $\ZD$-action on each fiber 
$\Phi^{-1}(c)_{sreg}$ 
does not in general extend to a global algebraic $\ZD$-action on $\XD$ that integrates the
$\fa$-action, but the lift $\hfa$ of $\fa$ does integrate to a global
 algebraic action on the covering
$\ghat$ by the previous theorem. In addition, we regard the formula for
the $\ZD$-action given in Equation (\ref{eq_zdaction}) as much simpler
than the formula given from \cite{Col1}.
\end{rem}

\subsection{Generic elements and irreducibility of $\ghat$}
\label{sec_ghatgeneric}

In this section, we show that $\ghat$ and $\XD$ are connected by considering
the open subset

$$\fzgen:= \{ (z_1, \dots, z_n) \in \fz_{\Dtower} :
\sigma_{i}(z_i) \cap \sigma_{i+1}(z_{i+1})=\emptyset, i=1, \dots, n-1 \},
$$

\noindent  its preimage $\ghatgen := \kappa^{-1}(\fzgen)$, and $\XDgen := \mu(\ghatgen)$.
We further show that when each $D_i$ consists of regular semisimple elements,
then $\mu:\ghatgen \to \XDgen$ specializes to a covering considered by
Kostant and Wallach in \cite{KW2}, and generalize a result in \cite{KW2}
concerning Hessenberg matrices to our setting.


Note that
\begin{eqnarray}
\ghatgen = \{(x, \uz): x\in \XDgen, (\uz) \in \fzgen \},\nonumber\\ \,
\XDgen=\{x\in\XD:\; \sigma_{i}(x_{i})\cap \sigma_{i+1}(x_{i+1})=\emptyset\} .\nonumber
\end{eqnarray}
It follows from definitions that
 $\ZD$ acts on $\ghatgen$.

\begin{cor}\label{c:ZDgen}
For $(\uz) \in \fzgen$, the group $\ZD$ acts simply transitively on the fibers 
$\kappa^{-1}(\uz)$.
\end{cor}

\begin{proof} 
Since the cardinality of the sets $\sigma_i(z_i) \cap \sigma_{i+1}(z_{i+1})$
is zero for $i=1, \dots, n-1$, the corollary follows by Theorem \ref{thm:ZDacts}.
\end{proof}

\begin{prop}
\label{prop:generichat}
The map $\mu: \ghatgen\to \XDgen$ is a $\Sigma_{\mathcal{D}}$-covering 
and $\ghatgen$ is smooth and irreducible.
\end{prop}

\begin{proof}
By Theorem \ref{thm_ghatcover}, 
 $\mu:\ghatgen \to \XDgen$ is a $\Sigma_{\mathcal{D}}$-covering of smooth
varities, so
it suffices to show 
that $\ghatgen$ is connected. By Remark \ref{ghatxdzsq}, 
$\kappa:\ghatgen \to \fz_{\Dtower, gen}$ is a surjective submersion,
so $\kappa$ is smooth of relative dimension $n^2 - \dnone$ (\cite{Ha}, Proposition III.10.4) and hence
flat (\cite{Ha}, Theorem III.10.2). Hence by Exercise III.9.1
in \cite{Ha}, it follows that $\kappa$ is an open morphism. By
Corollary \ref{c:ZDgen}, $\kappa$ has connected fibers, and it
is clear that $\fz_{\Dtower, gen}$ is connected. The Proposition
now follows from the following easy fact: if $f:X \to Y$ is
a surjective, open morphism, and $Y$ and all fibers of $f$ are
connected, then $X$ is connected.
\end{proof}

\begin{thm}
\label{thm_xdghatirr}
The varieties $\ghat$ and $\XD$ are connected and irreducible.
\end{thm}

\begin{proof}
By Theorems \ref{thm_XDsmooth} and \ref{thm_ghatcover},
and the fact that $\mu:\ghat \to \XD$ is surjective, it suffices to
prove that $\ghat$ is connected. Let $\ghat = Y_1 \cup \dots \cup Y_k$
be the connected components of $\ghat$. By Proposition \ref{prop:generichat},
we may assume that $\ghatgen \subset Y_1$. Hence, if $i > 1$, $Y_i \subset
\ghatsing := \ghat - \ghatgen$. By the dimension assertion in Theorem
\ref{thm_ghatcover}, it suffices to prove that $\dim(\ghatsing) < \dim(\ghat)$.

For this, since $\kappa:\ghat \to \fzD$ is smooth of relative dimension
$n^2 - \dnone$, for any locally closed subvariety $Y \subset 
\fzD$, $\dim(\kappa^{-1}(Y)) = \dim(Y) + n^2 - \dnone$ by Proposition III.10.1(b)
of \cite{Ha}. 
Let $\fzsing := \fzD - \fzgen$, and note that 
$\ghatsing = \kappa^{-1}(\fzsing)$. It follows easily 
 that $\dim(\ghatsing) = \dim(\fzsing) + n^2 -
\dnone$ and $\dim(\ghat) = \dim(\fzD) + n^2 - \dnone$.
 At the beginning of Section \ref{sec_hamcenter}, 
we identify $\C^{r_i} \cong \fz_i$, and we can then identify the variety $\fzD$
  with an open subset in $\C^r \cong \sum_{i=1}^{n} \fz_i$, $r=\sum_{i=1}^n r_i$.
  Using this identification, $\fzsing$ is identified with
  an open subset of a union of hyperplanes in $\C^r$. 
 It follows that $\dim(\fzsing) < \dim(\fzD)$, so
$\dim(\ghatsing) < \dim(\ghat)$.
\end{proof}

\begin{rem}
\label{rem_connecttokw}
Corollary \ref{c:ZDgen} and Proposition \ref{prop:generichat} specialize to give results
proved by Kostant and Wallach in \cite{KW2}. 
Indeed, suppose that each $D_i$ consists of regular semisimple elements.
Then Proposition \ref{prop:generichat} implies that 
$$
\ghatgen \cong M_{\Omega}(n, \fe ):= \{(x, z_1, \dots, z_n): (\uz)\in\fzgen,\, x_i \in G_i \cdot z_i, i=1, \dots, n \}
$$

\noindent is a covering of $\XDgen$ by a product of symmetric groups, which
is Theorem 4.14 in \cite{KW2}. Further, in this case $\ZD\cong(\C^{\times})^{\dn}$,
and the algebraic action of $\ZD$ on $\ghatgen$ is easily seen to coincide
with the algebraic action of $(\C^{\times})^{\dn}$ in \cite{KW2} which lifts
the $\fa$-action on $M_{\Omega}(n)$.
\end{rem}


We use Corollary \ref{c:ZDgen} to obtain an analogue of Theorem 5.12 
in \cite{KW2}.  In Section \ref{sec_kwresults}, we introduced 
the variety of upper Hessenberg matrices $\fb+e$ and noted 
that the Kostant-Wallach map restricts to an isomorphism of varieties
 $\Phi: \fb+e\to \C^{{n+1\choose 2}}$ and that $\fb+e\subset\fg_{sreg}$.  We define a closed subvariety of $\ghatgen$:
\begin{equation}\label{eq:ghathess}
\ghathess=\{(g_{1}, y^{1},\dots, g_{n}, y^{n})\in\ghatgen : \Ad(g_{n})y^{n}\in \fb+e\}
\end{equation}

\begin{thm}\label{thm:Hess}
 The morphism $\phi: \ZD\times \ghathess\to \ghatgen$ given by 
$(k, x)\to k\cdot x$, $k\in \ZD,\, x\in\ghathess$ is an isomorphism of algebraic
varieties.   Thus, $\ghatgen$ is a Zariski trivial $\ZD$-principal bundle over $\ghathess$.   
\end{thm}

\begin{proof}
By Proposition \ref{prop:generichat} and Zariski's main theorem (see \cite{TY} Corollary 17.4.8), it
suffices to show that $\phi$ is bijective. By Theorem \ref{thm:ZDacts} (1), the
morphism $\phi$ is injective. Let $(x,z_1, \dots, z_n) \in \ghatgen$, so $(\uz)
\in \fzgen$, and
let $c = \Phi(x)$.  There is a unique $\tilde{x}
\in \Phi^{-1}(c)_{sreg} \cap (\fb + e)$.
Since $\mu:\kappa^{-1}(z_1, \dots, z_n) \to \Phi^{-1}(c)_{sreg}$ is
an isomorphism by Equation (\ref{eq:kappafibre}), it follows that
$(\tilde{x}, z_1, \dots, z_n) \in \ghathess$.
Since $x$ and $\tilde{x}$ are elements of $\kappa^{-1}(\uz)$, by Corollary
\ref{c:ZDgen} there exists $k\in \ZD$ such that $k\cdot \tilde{x} = x$,
so $\phi$ is surjective.
\end{proof}

\begin{rem}\label{r:hesshatsmooth}
By Theorem \ref{thm:Hess}, the 
variety $\ghathess$ is a smooth and irreducible closed subvariety of $\ghatgen$.  
Moreover, the projection $\kappa: \ghatgen\to\fzgen$ restricts to 
an isomorphism of varieties $\ghathess \to \fzgen$. This last assertion
can be proved using the argument from the proof of the last theorem.
\end{rem}

\bibliographystyle{amsalpha.bst}

\bibliography{bibliography}

\end{document}